\documentclass[english,reqno]{amsart}
\usepackage[T1]{fontenc}
\usepackage[latin9]{inputenc}
\usepackage[letterpaper, margin=4.1cm]{geometry}
\usepackage{amsthm}
\usepackage{amssymb}
\usepackage{setspace}
\usepackage{enumerate}
\usepackage{fancyhdr}
\usepackage{xcolor}
\usepackage{upgreek} 
\usepackage{mathabx}
\usepackage[colorlinks=true]{hyperref}
\hypersetup{linkcolor=red,citecolor=blue,filecolor=dullmagenta,urlcolor=blue}

\usepackage{cancel}

\numberwithin{equation}{section}
\theoremstyle{plain}
\newtheorem{thm}{Theorem}[section]
\newtheorem{lem}[thm]{Lemma}
\newtheorem{prop}[thm]{Proposition}
\newtheorem{cor}[thm]{Corollary}
\newtheorem{claim}[thm]{Claim}
\newtheorem{rem}{Remark}[section]

\newcommand{\bb}[1]{{\mathbb{#1}}}
\newcommand{\ca}[1]{{\mathcal{#1}}}

\newcommand{\Orm}{ \left(\partial_r+\frac{2}{r}\right)}

\newcommand{\R}{\mathbb{R}}
\newcommand{\C}{\mathbb{C}}

\newcommand{\I}{\mathcal{I}}

\newcommand{\J}{\mathcal{J}(t)}

\newcommand{\al}{\alpha}
\newcommand{\bal}{\boldsymbol{\alpha}}

\newcommand{\bua}{\boldsymbol{u}_1}
\newcommand{\bub}{\boldsymbol{u}_2}
\newcommand{\bbb}[1]{\boldsymbol{#1}}

\newcommand{\uu}{\boldsymbol{u}}
\newcommand{\bbf}{\boldsymbol{f}}
\newcommand{\bbg}{\boldsymbol{g}}

\newcommand{\dt}{\frac{d}{dt}}

\newcommand{\Wa}{W_1}
\newcommand{\Wb}{W_2}

\newcommand{\vA}{\varphi}

\newcommand{\sech}{\operatorname{sech}}

\begin{document}

	\title[Decay of solutions of nonlinear Dirac equations]{Decay of solutions of nonlinear Dirac equations}
	\author[S. Herr]{Sebastian Herr} 
	\author[C. Maul\'en]{Christopher Maul\'en} 
	\address{Fakultat f\"ur Mathematik, Universit\"at Bielefeld,  Postfach 10 01 31, 33501 Bielefeld, Germany.}
	\email{herr@math.uni-bielefeld.de}
	\email{cmaulen@math.uni-bielefeld.de}
	\thanks{S.H. and Ch.Ma.: Funded by the Deutsche Forschungsgemeinschaft (DFG, German Research Foundation) -- Project-ID 317210226 -- SFB 1283}
	\author[C. Mu\~noz]{Claudio Mu\~noz}  
	\address{Departamento de Ingenier\'{\i}a Matem\'atica and Centro
de Modelamiento Matem\'atico (UMI 2807 CNRS), Universidad de Chile, Casilla
170 Correo 3, Santiago, Chile.}
	\email{cmunoz@dim.uchile.cl}
	\thanks{Cl.Mu.: Partially funded by Chilean research grants FONDECYT 1231250 and Basal CMM FB210005.}


	\begin{abstract}
	
	We study the long-time behavior of small and large solutions to a broad class of nonlinear Dirac-type equations. Our results are classified in 1D massless and massive cases, 3D general and $n$-dimensional in generality. In the 1D massless case, we prove that any globally defined solution converges to zero as time tends to infinity, within a spatial region expanding at a rate proportional to $ t \log^{-2} t$. This result holds without assumptions on the smallness of initial data or specific power of nonlinearity, ruling out the existence of standing breather-like or solitary wave structures in this regime. In the 1D massive case, solitary waves are known to exist. Introducing new virial identities adapted to Dirac's distinctive algebra, we prove that there are ``holomorphic'' odd nonlinearities under which globally defined small odd solutions decay to zero on spatial compact sets as time tends to infinity. This result is extended to the 3D case under boundedness of the $H^1$ norm but without requiring the parity condition on the data, giving decay proofs for an important class of nonlinear Dirac models, and opening the door to the future use of virial identities to prove asymptotic stability of well-chosen Dirac solitary waves. 

Finally, in higher dimensions $ n \geq 1$, we prove the $L^2$ decay for global solutions of nonlinear Dirac equations in the ``exterior light-cone'' region. This confirms the non-existence of breathers and other solutions propagating faster than the speed of light. Our proofs rely on carefully constructed weighted virial identities.
	\end{abstract}
	\maketitle
	\tableofcontents

	\section{Introduction}

\subsection{Setting} The (linear) Dirac equation was introduced as a relativistic version of the Schr\"odinger equation \cite{Dirac}, being one of the most relevant models in relativistic quantum mechanics. It describes the self-interaction of Dirac fermions \cite{Dirac,Thaller}. Several authors have introduced and justified a nonlinear version of this model, see e.g., Soler and Thirring models \cite{soler, T58,ARSV}. Generally, the nonlinear Dirac model is 
\begin{equation}\label{eq:dirac_covariant_form}
\begin{aligned}
    -i\gamma^{\mu}\partial_\mu \psi +m\psi=g(\psi\gamma^0 \psi)\psi,\qquad \psi(0,x)= \psi_{0}(x), \quad  x\in \R^{n}.
\end{aligned}
\end{equation}
Here $\psi:\R\times\R^{n}\to \C^{N}$ is a spinor-valued wave function,  $m\in \R$ is the mass, and $N=2^{\lfloor(n+1)/2\rfloor}$, where the summation convention is 
 $\gamma^{\mu}\partial_{\mu}=\sum_{\mu=0}^{n} \gamma^\mu \partial_{\mu}$, and $\partial_{0}\equiv \partial_t$. The Dirac matrices $\gamma^{\mu}$ are chosen such that they satisfy the anti-commutativity property
 \[
 \gamma^\mu \gamma^\nu+\gamma^\nu \gamma^\mu=2 \eta^{\mu \nu},
 \]
and where $\eta^{\mu \nu}$ is related to the Minkowski metric $\eta=\mbox{diag}(1,-1,\dots,-1)$. Some members of the Dirac family equation are invariant under Lorentz boosts, e.g., the Soler and Thirring models.

Let $\ca{D}_m=-i\gamma^{\mu}\partial_\mu+m$ be the Dirac operator which is called massive if $m\ne 0$ resp.\ massless if $m=0$. Let $\beta:=\gamma^0$ and  $\gamma^0 \ca{D}_m=-i\partial_t +\ca{H}_m$. The corresponding Hamiltonian is $\ca{H}_m= -i\beta\gamma^{j}\partial_j +m\beta$, and the operator $\mathcal{H}=\mathcal{H}_0$ is given by
\begin{equation}\label{H}
\mathcal{H}=- i\bal \cdot\nabla=-i\alpha^j\partial_j.
\end{equation}
Here $\bal=(\alpha^1,\dots,\alpha^n)$, with the matrices $\alpha^j= \beta \gamma^j$, see \eqref{eq:alpha_matrix}.  

This work is concerned with a \emph{nonlinear Dirac equation} posed in $n\geq 1$ dimensions.   
The nonlinearity will be more general than in \eqref{eq:dirac_covariant_form}. To encompass the results proved in this work, we shall assume the existence of a (possibly) nonlinear function $V:\R\times\C^{N} \to \R$ and $\psi$ solution to
\begin{equation}\label{eq:dirac_1}
\begin{aligned}
i\partial_{t} \psi =&~{}\mathcal{H} \psi+ (m+V(x,\psi)) \beta \psi\\
\psi(0,x)=&~{}\psi_{0}(x),
\end{aligned}
\end{equation} 
The above system enjoys at least three conserved quantities \cite{SV}. These are: the 
charge 
\begin{equation}
Q[\psi]=\int \psi^{\dagger}\psi \label{eq:charge},
\end{equation}
and for the special case $V(\psi)=-g(\overline{\psi}\psi)$, the associated energy
\begin{equation}
\begin{aligned}
E[\psi]=& \int \left( \psi^{\dagger}\mathcal{H} \psi+m\psi^{\dagger}\beta\psi-G(\psi^{\dagger}\beta\psi) \right),\quad G(s) =\int_0^s g, \label{eq:energy}
\end{aligned}
\end{equation}
and the Lagrangian
\begin{equation*}
\begin{aligned}
L[\psi,\partial_t \psi]=&~{} \mbox{Im} \int  \psi^{\dagger} \psi_{t}
-E[\psi]
= \int [G(\overline{\psi}\psi)-g(\overline{\psi}\psi)\overline{\psi}\psi],
\end{aligned}
\end{equation*}
(here $\int =\int_{\R^n}$,) and where $\overline{\psi}=\psi^{\dagger}\beta$, and $\psi^{\dagger}$ is the complex conjugate transpose of the vector $\psi$. 

\subsection{Well-posedness}
The well-posedness problem of the nonlinear Dirac equation is challenging due to the lack of a defined sign on the Hamiltonian. Most works address the massive and the massless cases separately. The literature is extensive across both 1D and higher dimensions; here, we highlight the most relevant works.

In the 1D case, Candy \cite{C11} used null coordinates to prove global existence in $H^s$ for all $s \geq 0$ in the case of the $L^2$-critical and integrable massive Thirring model. Pelinovsky \cite{Pelinovsky_Survey} showed the existence of small-$H^1$-norm solutions for more general nonlinear Dirac equations with cubic and higher-order nonlinear terms, and also established scattering for global solutions with small initial data in $H^1 \cap L^1$. In \cite{MNT} the authors proved local well-posedness in $H^s$ with $s > -1/2$ and ill-posedness in $H^{-1/2}$ for the quadratic nonlinear Dirac equation. See \cite{C11, Pelinovsky_Survey, MNT,PSa,S,BH,BC,ST} and references therein for further details.

In the 3D case, Escobedo and Vega \cite{EV97} proved local and global well-posedness in $H^s$ with $s > 1$ for the massive nonlinear Dirac equation with covariant nonlinearities, i.e., nonlinearities that preserve Lorentz invariance. In the massless case, Tzvetkov \cite{T98} established global existence for small smooth initial data, assuming $|V| \lesssim |\psi|^p$ with $p > 2$.

For the cubic Dirac equation in 2D and 3D, Candy and the first author \cite{CH23} employed a bilinear Fourier restriction method and atomic function spaces to prove the global well-posedness of the Cauchy problem for small initial data. This approach provides a unified treatment of the massive and massless cases, showing their intrinsic connections and obtaining convergence in the massless and the non-relativistic limit.

As far as we know,  there are just a few works that address the ill-posedness and blow-up. In \cite{DO16}, the authors demonstrated ill-posedness and blow-up for both large and small data in the case of \emph{non-gauge invariant nonlinearities}. Additionally, in \cite{MNT} it was established that for any $s < 0$, the flow map $\psi_0 \mapsto \psi$ from $H^s(\mathbb{R}^3)$ to $C([0,T]; H^s(\mathbb{R}^3))$ is not locally of class $C^3$. Finally, in the 1D case, \cite{HP_SSBU} proved the nonexistence of self-similar blow-up solutions in the space of bounded functions.

\subsection{Main results}

In this paper, our main goal is to study the long-time asymptotic behavior of globally defined nonlinear Dirac solutions, under minimal assumptions, and including integrable and nonintegrable models. Highly relevant for us will be the 1D and 3D cases. We first consider the 1D Dirac equation in laboratory coordinates:
\begin{equation}\label{eq:D_LC}
\begin{aligned}
i (\partial_t u+ \partial_xu)+mv=&~{} \partial_{\overline{u}} W_1(u,\bar u, v, \bar v)\\
i (\partial_t v- \partial_x v)+mu=&~{}\partial_{\overline{v}} W_2 (u,\bar u, v, \bar v),
\end{aligned}
\end{equation}
 where $W_1,W_2:\C^2\to \R$ represent nonlinearities and satisfy the following conditions: for each $i,j=1,2,$
 \begin{enumerate}[(a)]
 \item\label{en:1} $W_j(u,v)=W_j(v,u)$ 
 \item\label{en:2} $W_j(e^{i\theta} u,e^{i\theta} v)=W_j(u,v)$ for any $\theta \in \R$.
 \item\label{en:3} $W_j$ is polynomial in $(u,v)$ and $(\overline{u},\overline{v})$.
 \end{enumerate}

In the case $W_1=W_2=|u|^2|v|^2$, the model is integrable \cite{T58,CP_Block,Pelinovsky_Survey}. Moreover, by normalizing the mass $m=1$ and considering $|\omega|<m=1$, specific stationary solitary waves are present, of the form 
\begin{equation}\label{sola}
\begin{aligned}
& u= U_\omega (x+x_0) e^{i\omega t + i\alpha}, \quad v= \overline{U}_\omega (x+x_0) e^{i\omega t + i\alpha},\\
& U_\omega(x)= \frac{\gamma}{\sqrt{1+\omega} \cosh(\gamma x) +i \sqrt{1-\omega} \sinh(\gamma x)}, \quad \gamma:=\sqrt{1-\omega^2}.
\end{aligned}
\end{equation}
These (standing) solitary waves are a clear obstruction to decay, having interesting symmetries. Indeed, one can see that among the symmetries preserved by \eqref{eq:D_LC} one can find the (even + $i$ odd) and (odd + $i$ even) ones, which are precisely present in the solitary waves \eqref{sola}. Moreover, notice that the classical (odd + $i$ odd) symmetry, natural in scalar field models with odd nonlinearities, is not preserved this time. This makes the exercise of finding decay in the energy space in the 1D case a highly complicated task. Indeed, we shall concentrate efforts in another manifold of initial data not producing solutions such as \eqref{sola}. Within this region, we shall obtain energy local decay. 

Some remarks are in order. Let $T:=i \begin{pmatrix} -1&1\\~{}i&i \end{pmatrix} $.
 
\begin{rem}\label{rem1p1}
If\footnote{Notice that this transformation is slightly different from the one used in \cite{Pelinovsky_Survey} to be able to consider solitary wave data $e^{iwt} \phi_w$, $\phi_w=(\phi_1,\phi_2)\in \R^2$ (real-valued), with even-odd symmetry. For more details see \cite{Pelinovsky_Survey,CP_Block,PS_AS,CPS} and \cite{BC_1d,CVPS}. } $\psi=T\uu$ and $\mathcal{H}= -i\alpha^1\partial_x$ (see Section \ref{Pauli_matrices}), from \eqref{eq:D_LC} one gets a special case of \eqref{eq:dirac_1} in the regime $W_1=W_2=W$:
\begin{equation}\label{eq:D_1d}
\begin{aligned}
i\partial_t \psi_1 =&~{}\partial_x \psi_2+m\psi_1-\widehat W_1,
\\
i\partial_t \psi_2 =&~{}-\partial_x \psi_1-m\psi_2+ \widehat W_2,
\end{aligned}
\end{equation}
where $(\psi_1,\psi_2)\in \C^2$, $\widehat W_1=i(\partial_{\overline{v}} W-\partial_{\overline{u}} W)(T^{-1}\psi)$ and $\widehat W_2=-(\partial_{\overline{v}} W+\partial_{\overline{u}} W)(T^{-1}\psi)$. For more details see \cite{Pelinovsky_Survey,PS,PS_AS,CPS}.
\end{rem}

\begin{rem}
Let $\Re$ and $\Im$ denote the real and imaginary part operators acting on complex numbers. Then \eqref{eq:D_LC} has at least the following three conserved quantities
 \[
 \begin{aligned}
 H(u,v)=&~{}\Im \int (u_x \overline{u}-v_x \overline{v})dx+\Re\int( u\overline{v} -W(u,v) ) dx,\\
 P (u,v)=&~{}  \Im \int  (u\overline{u}_x+v \overline{v}_x) dx, \qquad 
 Q (u,v)= \int (|u|^2+|v|^2)dx,
 \end{aligned}
 \]
 corresponding to the Hamiltonian, momentum, and charge, respectively.
\end{rem}

It is well-known that the cubic Nonlinear Schr\"odinger (NLS) equation admits \emph{breather solutions}. These solutions are characterized by being nontrivially \emph{spatially localized} and \emph{periodic in time}, distinguishing them from standing waves. Breather solutions possess a more intricate analytical structure; for example, they may involve complex interactions between different modes, which result in oscillatory behavior that is both spatially and temporally localized. One of the well-known breathers for the focusing cubic NLS is the Satsuma-Yajima breather solution (see \cite{AFM21,AC_NLS_NonExistence} and the reference therein). For the Klein-Gordon equation (and for a class of 1D nonlinear wave equations), in \cite{KMM_Nonexistence_KG}, the authors proved the nonexistence of small odd breather solutions (see also \cite{SW}).  

Similarly, Mart\'inez \cite{EM20} showed that, for a wide range of nonlinearities, any global odd solution to the 1D NLS equation decays to zero in compact regions of space as time tends to infinity, assuming only initial data in the energy space. Previous work addressed the decay/scattering problem under the condition that the power $p$ satisfies $3<p<5$, but this result improved the range, including $1< p\leq3$. More recently, \cite{AC_NLS_NonExistence} characterized the nonexistence of breather solutions for the  $n$-dimensional NLS using the energy and power of the corresponding nonlinearity.

Given the similarities between the Dirac equation and both the Klein-Gordon and the classical Nonlinear Schr\"odinger equations, a natural question arises: \emph{Do breather solutions exist for the Dirac equation?} Despite its deep connection to the Klein-Gordon and Schr\"odinger equations, the Dirac system has a unique structure. Energy-based methods cannot be directly applied because the energy functional is not sign-definite. Furthermore, the Hamiltonian of the nonlinear Dirac equation is unbounded from above and below, reflecting the presence of both particles and antiparticles in the model \cite{Thaller}. Consequently, the energy does not provide a priori estimates. Concerning the existence of breather solutions in the 1D massless Gross--Neveu model, for large values of $N$, the existence of breather solutions under a symmetric scalar-scalar interaction was proved in \cite{DHN}. In this context, breather solutions are interpreted as multi-fermion bound states oscillating in time at their rest frame. Recently, in \cite{FT}, the authors described a method to find such solutions and studied the breather-breather scattering problem.

\medskip

Let us introduce the time-dependent interval
\begin{equation}\label{eq:Ibt}
I(t):=\left(	-\frac{|t|}{\log^2|t|},\frac{|t|}{\log^2|t|}\right), \quad |t|\geq 10. 
\end{equation}

Our first result describes decay of global solutions of the \emph{massless 1D nonlinear Dirac equation}. 

\begin{thm}\label{thm:massless}
Let $(u,v)\in C^1_{loc}(\R: (L^2\times L^2)(\R;\C))$ be any global solution of the massless 1D Dirac equation \eqref{eq:D_LC} with $W_1=W_2=W$ such that \eqref{en:1}-\eqref{en:2}-\eqref{en:3} are satisfied.  Then, for $I(t)$ as in \eqref{eq:Ibt},
\begin{equation}\label{limit}
\lim_{t\to \infty} \|(u,v)(t)\|_{L^2(I(t))}=0.
\end{equation}
Therefore, no soliton nor breather solution exists for the Dirac equation \eqref{eq:D_1d} inside the region $I(t)$ as time is large enough. 
\end{thm}

In particular, Theorem \ref{thm:massless} holds for any compact interval. This result complements previous findings obtained in \cite{EM20} for the 1D NLS under odd data. Indeed, the Dirac equation is closely related to both the Schr\"odinger and the Klein-Gordon equations, and the Schr\"odinger equation emerges in the non-relativistic limit, which roughly corresponds to taking $m\to \infty$.

It is interesting to discuss why Theorem \ref{thm:massless} holds, but it does not hold in the massive case. First of all, a simple dispersive analysis reveals that in the case $m=0$ linear waves escape to infinity with speed one, meaning that compact intervals of time remain free of these slowly decaying waves. This is not the case in the massive setting, which is more related to NLKG in that sense. We have found a particular new structure in the massless case that allows us to cover almost the whole light cone, with some loss coming from integrability estimates. This will not be the situation in the massive case.

The \emph{decay properties} of solutions to the linear Dirac equation, both with or without a potential, have been studied extensively \cite{DAF,DAF2,K,BG_1dP, BDAF,Cacciafesta,CFK,DAFS}. For the 1D linear massive Dirac equation with a potential, \cite{K} established dispersive long-time decay in weighted $L^2$ norms for solutions, even in the presence of a generic potential without imposing smallness conditions on it. Furthermore, $L^1\to L^\infty$ dispersive estimates were proved in  \cite{BG_1dP}, showing that the natural $t^{-1/2}$ decay rate can be improved to $t^{-3/2}$ using weighted spatial norms.

\medskip

The natural corollary of Theorem \ref{thm:massless} in terms of the variable $\psi$ is as follows:

\begin{cor}\label{cor:mass_less}
Let $(\psi_1,\psi_2)$ be any global solution in $ C^1_{loc}(\R; (L^2\times L^2) (\R;\C)$ of the massless Dirac equation given in \eqref{eq:D_1d}. Then, for $I(t)$ as in \eqref{eq:Ibt}, there is a strong decay to zero on the $L^2$ norm. Therefore, no soliton nor breather solution exists for the massless Dirac equation \eqref{eq:D_1d} inside the region $I(t)$, for any time t sufficiently large.
\end{cor}

The proof of the Corollary \ref{cor:mass_less} follows from \eqref{limit}, Remark \ref{rem1p1} and noticing that
\[
\|\psi_1\|^2_{L^2(I(t))}=\|u- v\|^2_{L^2(I(t))}\leq 2\left(\|u \|^2_{L^2(I(t))}+\|v\|^2_{L^2(I(t))} \right),
\]
\[
\|\psi_2\|^2_{L^2(I(t))}=\| u+v\|^2_{L^2(I(t))}\leq 2\left(\|u \|^2_{L^2(I(t))}+\|v\|^2_{L^2(I(t))} \right).
\]
Let us consider now the 1D massive case. Here we know the existence of standing waves (see \eqref{sola}), enemies to decay estimates, therefore a different approach is needed. To fix ideas and profit of the Dirac's algebra, let us consider Dirac in the form
\begin{equation}\label{eq:D_1d_new}
\begin{aligned}
i\partial_t \psi_1 =&~{}\partial_x \psi_2+m\psi_1- \Wa,
\\
i\partial_t \psi_2 =&~{}-\partial_x \psi_1-m\psi_2+ \Wb.
\end{aligned}
\end{equation}
We shall assume conditions on the nonlinearity that emerges from $W_1,W_2$ 
that avoids the solitary wave manifold:
\begin{itemize}
\item Let $(\Wa,\Wb)=({\partial_{\overline{\psi}_1}W},{\partial_{\overline{\psi}_2}W})$ be odd and of polynomial type, for some polynomial $W$.
\item For $j=1,2$, if we write $\Wa=\Wa (a,b,c,d)$ and $\Wb=\Wb (a,b,c,d)$, with $(a,b,c,d)=(\psi_1,\bar\psi_1,\psi_2,\bar\psi_2)$, then $W_1$ and $W_2$ are ``harmonic'', in the sense that
\begin{equation}\label{Laplacian}
 \qquad \partial_a \Wb +\partial_c \Wa =  \partial_b \Wb - \partial_d \Wa=\partial_c \Wb -\partial_a \Wa=\partial_d \Wb +\partial_b \Wa =0.
 \end{equation}
 \item Finally, $W_1$ and $W_2$ only depend on $b$ and $d$:
 \begin{equation}\label{W_dependence}
 W_1=W_{1,0}(b,d),\quad  W_2=W_{2,0}(b,d).
 \end{equation}
\end{itemize}

The condition \eqref{Laplacian} is a ``Cauchy-Riemann'' type condition, and will ensure that solutions issued from odd data will continue having that property for late times, even if \eqref{eq:D_1d_new} seems do not respect that property in general. Therefore, the standard integrable nonlinearities and the odd-even or even-odd parities, naturally preserved by Dirac's model, and the origin of the existence of non-decaying solitary waves, will be discarded by this condition. As a consequence, one will have that $i\partial_t \psi_1-\partial_x \psi_2$ and $i\partial_t \psi_2 + \partial_x \psi_1$ will be odd for all times. Additionally, under better smoothness conditions \eqref{Laplacian} implies the natural condition
\begin{equation}\label{Laplacian2}
(\partial_a^2 +\partial_c^2)W_j =(\partial_b^2 +\partial_d^2)W_j =0, \quad j=1,2,
\end{equation}
giving the reason why we adopted the name ``harmonic'' (or Cauchy-Riemann related) property. Finally, condition \eqref{W_dependence} is necessary to get ``conservation'' of a suitable Lyapunov functional, meaning that more general nonlinearities may not be able to retain boundedness in time of solutions. For examples of nonlinearities satisfying these conditions, see Section \ref{sec:appl}. 
For this model, we have the following decay property:

\begin{thm}\label{thm:massless2}
Let $(\psi_1,\psi_2)\in C^1_{loc}(\R: (L^2\times L^2)(\R;\C)) \cap C_{loc}(\R: (H^1\times H^1)(\R;\C))$ be any small global solution of the massive 1D Dirac equation \eqref{eq:D_1d_new} with odd harmonic nonlinearity $(W_1,W_2)$ in the sense of \eqref{Laplacian} and \eqref{W_dependence}. Assume that $(\psi_1,\psi_2,\partial_t\psi_1,\partial_t\psi_2)(t=0)$ are odd. Then $(\psi_1,\psi_2)$ are odd for all times and for any compact interval $I$, 
\begin{equation}\label{limit2}
\lim_{t\to \infty} \|(\psi_1,\psi_2)(t)\|_{(L^2\cap L^\infty)(I)}=0.
\end{equation}
Therefore, no standing soliton nor breather solution exists for the 1D massive Dirac equation if time is large enough. 
\end{thm}

The proof of Theorem \ref{thm:massless2} involves the introduction of new virial identities adapted to the Dirac setting, without requiring to square of the equation to get a NLKG model. Indeed, performing this classical trick may be more subtle than usual, see Appendix \ref{AppB} for details. Indeed, it is proved in Lemma \ref{trick} that the classical process of ``squaring the equation'' may not be completely rigorous unless one assumes more conditions on the data and the corresponding solutions. In our case, estimates are obtained in the purely Dirac form \eqref{eq:D_1d_new}, recovering the dynamics of $\psi_1$ and $\psi_2$ at the same time. We believe that virial estimates of this type are new and introduce an interesting direction where one can further investigate in the future, in particular in the case of manifolds around solitary wave solutions.

\medskip

Now we consider the 3D case. For the 3D magnetic linear Dirac equation, a potential-perturbed version of the linear Dirac equation, but not considered in this work, the authors in \cite{BDAF} derived a virial identity and subsequently used it to obtain smoothing and endpoint Strichartz estimates. Additionally, they proved a Hardy-type inequality for the perturbed Dirac operator. Due to the lack of a definite sign in the Dirac operator, the authors addressed this challenge by considering the \emph{squared Dirac equation}. Later, in \cite{Cacciafesta}, the author extended these results to dimensions higher than three. The aforementioned works relied on a virial identity, which leveraged the hidden Klein-Gordon structure. Furthermore, the Hamiltonian  $\ca{H}_m$, derived from the Dirac operator  $\ca{D}_m$, is unitarily equivalent to the square root of the Klein-Gordon Hamiltonian $\gamma^0 \sqrt{-\Delta+m^2}$. In other words, it satisfies  $\ca{H}_{m}^2=(-\Delta+m^2)I_{N}$ (see \cite{Thaller, CH23, Cacciafesta,BDAF}, and the references therein.)
Finally, see \cite{CS_2019} for results on long-time behavior of solutions in the case of linear and nonlinear Dirac models on manifolds. 

\medskip

In this paper, we will consider the following 3D nonlinear Dirac equation
\[
i\partial_{t} \psi =~{}\mathcal{H} \psi+ (m-g(\psi,\overline{\psi})) \beta \psi, \quad \psi \in \C^4. 
\]
Now, in view of the possible existence of standing waves, we shall focus on a particular symmetry condition.
To understand this point, we discuss the more general case of the so-called partial wave decomposition. There exists an orthogonal decomposition
\[
L^2(\mathbb{S}^2)^4 \cong \bigoplus_{j,m_j,k_j} \mathcal{H}_{j,m_j,k_j}
 \cong \bigoplus_{j=\frac12, \frac32,\ldots}^{\infty} \, \bigoplus_{m_j=-j}^{j} \, \bigoplus_{k_j=\pm (j+\frac12)} \mathcal{H}_{m_j,k_j}
\]
where the spaces $\mathcal{H}_{m_j,k_j}$ are called partial wave subspaces, and each partial wave subspace is of dimension two. The spaces $\mathcal{H}_{m_j,k_j}$ is given by
\[
\mathcal{H}_{m_j,k_j} =\left\{ c^{+}\Phi_{m_j,k_j}^{+}+c^{-}\Phi_{m_j,k_j}^{-} ~\big| ~c^{\pm}\in \C \right\},
\]
with
\[
\Phi_{m_j,\mp(j+\frac12)}^{+}= \begin{pmatrix}
i \Psi_{j\mp \frac12}^{m_j}\\0
\end{pmatrix}
,
\quad 
\Phi_{m_j,k_j}^{-}=\begin{pmatrix}
0\\ \Psi_{j\pm \frac12}^{m_j}
\end{pmatrix}.
\]
(See \cite[Section 4.6]{Thaller}.)
In order to simplify the discussion, here we are just considering a subspace of dimension 1. This choice avoids the existence of bi-frequency solitary waves. In particular,  let us consider the case of ``spherical coordinates'' $(r,\theta,\phi)\in (0,\infty)\times (0,\pi)\times (0,2\pi)$; and let us consider a solution of the Dirac equation of the form (Soler \cite{soler})
\begin{equation}\label{eq:psi1}
\psi_{1}(t,x)=\begin{pmatrix}
u(t,r) \begin{pmatrix}
1\\0
\end{pmatrix}
\\
i v(t,r)
\begin{pmatrix}
\cos \theta\\
\sin \theta e^{i\phi}
\end{pmatrix}
\end{pmatrix}, \quad \mbox{ with } u,v \in \C.
\end{equation}
Notice that with this choice, we are working on a subspace of the partial wave space. Then, after some classical computations, one gets
\[
 \mathcal{H} \psi_1= \begin{pmatrix}
 \left(  \partial_r  v +\frac{2}{r}v \right)\begin{pmatrix}
1
 \\
0
 \end{pmatrix}
\\
-i\partial_r u \begin{pmatrix}
   \cos \theta  
   \\
e^{i\phi } \sin \theta 
 \end{pmatrix}
 \end{pmatrix}.
\]

Therefore, the 3D Dirac equation \eqref{eq:dirac_1} as a $4\times 4$ complex system with $V=-g(u,v)$ associated to this  partial wave subspace-type solution \eqref{eq:psi1} is given by
\begin{equation}\label{PW3}
\begin{aligned}
i\partial_t u=&~{} \left(  \partial_r  v +\frac{2}{r}v \right)+(m-g)u,
\\
i\partial_t v=&~{} -\partial_r u -(m-g)  v,
\end{aligned}
\end{equation}
that is a reduced $2\times 2$ complex system. 
It turns out that there exist standing waves for this model for a huge range of nonlinearities, see Merle \cite{Merle88} and references therein. Consequently, decay cannot hold on compact spaces in these cases. Therefore, our main result in the 3D case, Theorem \ref{thm:3D}, will be valid in a manifold of initial data preserved by the flow and where standing waves are not present. 
With a slight abuse of notation, we study decay in the invariant space that shares the same angular structure as in \eqref{eq:psi1} and following \eqref{PW3}. We then consider the generalized model
\begin{equation}\label{eq:PW3} 
\begin{aligned}
i \partial_t  \phi_1 =&~{}  \bigg(\partial_r+\frac{2}{r}\bigg) \phi_2 +m \phi_1-W_1
\\
i \partial_t \phi_2 =&~{} -  \partial_r \phi_1 -m\phi_2 +W_2.
\end{aligned}
\end{equation}
where $(\phi_1,\phi_2)=(\phi_{11}+i\phi_{12},\phi_{21}+i\phi_{22}) \in \C^2$, and $(W_1,W_2)=(W_{11}+iW_{12},W_{21}+iW_{22}) \in \C^2$.
Therefore, using $(\phi_1,\phi_2)$ in the above system, we construct the spinor that belongs to the partial wave subspace and  the non-linearity related to \eqref{eq:PW3}, i.e.
\begin{equation}\label{eq:phi}
\begin{pmatrix}
\phi_1(t,r) \begin{pmatrix}
1\\0
\end{pmatrix}
\\
i \phi_2(t,r)
\begin{pmatrix}
\cos \theta\\
\sin \theta e^{i\phi}
\end{pmatrix}
\end{pmatrix},
\quad \mbox{ resp. } \quad 
\begin{pmatrix}
W_1 \begin{pmatrix}
1\\0
\end{pmatrix}
\\
i W_2
\begin{pmatrix}
\cos \theta\\
\sin \theta e^{i\phi}
\end{pmatrix}
\end{pmatrix}.
\end{equation}
where for each $j=1,2$ we decompose  $W_j=W_{j1}+ i W_{j2}$, $W_{j1}, W_{j2}$ real-valued.

{
Additionally, consider that the nonlinearity is of pure power type. This assumption allows us to capture a wide range of nonlinearities with physical relevance. Let  $(W_1,W_2)$ be a nonlinearity such that
\begin{itemize}
\item There exists $C>0$ such that for $(W_1,W_2)=(W_1,W_2)(\phi_1,\overline{\phi}_1,\phi_2,\overline{\phi}_2)$
\begin{equation}\label{eq:WW_w}
|W_1|+|W_2| \leq C |\phi|^{p} \mbox{ with } p\geq 3.
\end{equation}
\end{itemize}
}

Having explained the meaning of \eqref{eq:WW_w}, we can now state our main result in the 3D case.

\begin{thm}\label{thm:3D}
Let ${\bf \phi}=(\phi_1,\phi_2)\in C^1_{loc}(\R: L^2(\R^3;\C)^2)\cap C_{loc}((\R: H^1(\R^3;\C)^2))$ be any radial global solution of the 3D Dirac equation \eqref{eq:PW3} (in the sense of \eqref{eq:phi}) with polynomial-type nonlinearity such that $W_{jk}$ satisfies \eqref{eq:WW_w}. 
Assume additionally that $\phi$ is small, in the sense that $\sup_{t\geq 0} \|{\bf \phi}(t)\|_{H^1\cap L^\infty} \ll 1$. Then, for any $R>0$,
\begin{equation}\label{limit3D}
\lim_{t\to \infty} \| {\bf \phi} (t)\|_{L^2(B(0,R))}=0.
\end{equation}
Therefore, no small soliton nor breather solution exists for the  Dirac equation \eqref{eq:PW3} inside the ball $|x|<R$.
\end{thm}

{
\begin{rem}[On Soler's nonlinearities]
 Recall that condition \eqref{eq:WW_w} includes classical models with solitary waves, such as the Soler's nonlinearities. Indeed, in the Soler's case one has
\begin{equation}\label{Soler}
W_1 = g(|\phi_1|^2-|\phi_2|^2)\phi_1, \quad W_2 =g(|\phi_1|^2-|\phi_2|^2)\phi_2,
\end{equation}
where $g$ is a polynomial with $g(0)=0$.  It follows that $(W_1,W_2)$ satisfies \eqref{eq:WW_w}. In other words, for Soler's nonlinearity, the above Theorem holds, implying that there are no arbitrarily small and localized solitary waves or breather solutions.
\end{rem}
}

By a classical Sobolev embedding, one also has decay of local Sobolev subcritical $L^p$ norms. Notice that no conditions on the parity of the data are required, and \eqref{Laplacian} is also not required in the 3D case. This is consistent with the fact that it is expected that breathers are difficult to obtain in dimensions larger than one. Mathematically speaking, decay motivated ''virial identities'' in the radial 3D case present better spectral properties than in the 1D case \cite{MoMu}, making more quantitative the classical argument ''the bigger the dimension, the better the linear decay''. This is reflected in the fact that one can cook up suitable virial weights under which local decay is achieved, and this procedure fails in 1D by classical counterexamples (breathers). In this paper, after a suitable decomposition of $\phi$ into components $(\phi_{11},\phi_{12},\phi_{21},\phi_{22})$, we will consider 3D radial virial functionals of the form 
\[
\begin{aligned}
\mathcal{K}_{1}= &~{} \int \left[ \vA  \partial_r \phi_{11} +\frac12 \vA ' \phi_{11} \right]\left( \Orm \phi_{22} + m \phi_{12} \right),\\
\widetilde{\mathcal{K}_{1}} =&~{} \int \left[ \vA  \partial_r \phi_{22} +\frac12 \vA ' \phi_{22} \right]\left( \partial_r \phi_{11} + m \phi_{21} \right),
\end{aligned}
\]
with $\vA  =\frac{r^{3/2}}{1+r}$, plus other two related virials $\mathcal{K}_{2}$ and $\widetilde{\mathcal{K}_{2}}$. The combination of these four functionals and the precise choice of the weight $\varphi$ with a strong decay at $r=0$ but an additional decay at infinity will help us to describe the long-time local in space dynamics of small globally bounded 3D Dirac massless and massive waves, allowing us to prove \eqref{limit3D}. 

\subsection{Solitary waves}

Theorem \ref{thm:3D} addresses the problem of decay in a subspace of the partial wave subspace where small solitary waves are not present.  To complement the previous literature, we discuss the case of coherent structures (see \cite{Thaller,Dudnikova}). 

The existence and stability of solitary waves are problems that have been widely studied; however, it is not completely well understood. The stability problem, in particular, has been extensively explored in both one-dimensional (1D) and three-dimensional (3D) cases. For the 1D case, notable results have been obtained concerning the asymptotic stability of solitary waves in certain Dirac-type models, often under specific restrictions on the types of perturbations considered.

For the 1D Dirac equation in \eqref{eq:D_1d}.  In \cite{BC_1d}, solitary wave solutions' existence and qualitative properties have been investigated for generic nonlinearities, giving an explicit form of solitary wave solutions to the massive Gross-Neveu model. Additionally, the authors analyzed the spectral stability of these solitary wave solutions, deriving explicit expressions for several eigenfunctions. In \cite{ARSV}, the spectral stability of the nonlinear Dirac operator was studied, with a particular focus on self-interacting nonlinearities. The authors derived bounds on the eigenvalues of the linearized operator around standing wave solutions. Specifically, for pure power nonlinearities, they identified a frequency range within which the linearized operator lacks real or purely imaginary unstable eigenvalues. For more details about solitary waves' spectral stability, see \cite{BC_Book,CBCKS_2018}.

Regarding asymptotic stability, Comech, Van Phan, and Stefanov \cite{CVPS} addressed the 1D nonlinear Dirac equation with scalar self-interaction involving quintic and higher-order nonlinearities, known as the Gross-Neveu model. They established the asymptotic stability of solitary waves for a range of frequencies. In their analysis, they considered ``even'' perturbations, excluding in this way translations and eigenvalues of the form $\pm 2wi$.

Furthermore, on the 1D Dirac equation in \eqref{eq:D_1d}. In \cite{CP_Block}, described the existence and explicit form of gap solitary waves solutions but also studied the spectral stability of such solutions. Pelinovsky and Shimabukuro \cite{PS},  for the Massive Thirring Model, proved that the solitary waves are orbitally stable on $H^1$. After, the same authors and Contreras \cite{CPS},  proved that the solitary waves are orbitally stable on $L^2$. The authors rely on the integrability of the Massive Thirring Model,  thanks to the inverse scattering transform they find a new conserved quantity and are able to use the auto-B\"acklund transformation. Regarding the asymptotic stability, in \cite{PS_AS} the authors proved dispersive decay estimates and used them to prove asymptotic stability of small gap solitons with quintic and higher-order nonlinear terms. For the 3D Dirac equation, the standing wave had been described using a symmetry decomposition (see \cite{BC12,SV,BCDM88,Merle88}), and the stability properties have been studied these solitary waves.

  Cazenave and Vazquez \cite{CV86} proved the existence of stationary states for the 3D nonlinear Dirac equation with a nonlinearity $-g(\overline{\psi}\psi)$. Following the Wakano ansatz \cite{wakano} and Soler \cite{soler}, they obtained solutions that are separable in spherical coordinates and exponentially decaying. This result was generalized for a wide variety of nonlinearities by Merle \cite{Merle88}, e.g., nonlinearities that are not necessarily increasing or bounded below their value in zero. 
  
  Furthermore, Balabane, Cazenave, Douady, and Merle, in \cite{BCDM88}, proved the existence of infinitely many stationary states, ordered by the number of nodes of each component. Later on, using variational techniques, Esteban and Ser\'e \cite{ES95} proved the existence of stationary solutions when the nonlinearities are non-compatible with symmetry reductions. For instance,  Boussaid and Cuccagna \cite{BC12} consider the stability problem for the standing waves of a class of massive nonlinear Dirac equations and, under some technical hypotheses, proved orbital and asymptotic stability.

   Recalling the connection between NLS and the nonlinear Dirac equation. We shall highlight that the Soler model shares some interesting but complex properties with the focusing cubic NLS. In the focusing cubic NLS, it has been proved that there is a family of multi-solitons which are arbitrarily small perturbations in $H^1$, of the solitary wave $\sqrt{2}\sech(x)$ (see \cite{Martel3th5th} and the reference therein). These features have been described by Boussa\"id and Comech in \cite{Bifrequency_BC} for the 3D Soler model. Furthermore, they showed that the bi-frequency solitary waves are associated with the Bogoliubov ${\bf SU}(N/2,N/2)$ symmetry, giving some lines about how it is the proper way to understand the asymptotic stability of these solutions. In \cite{BCN_Bi_2024},  the same authors and Kulkarni proved that bi-frequency solitary waves might also possess linear stability properties similar to those present in the one-frequency solitary wave.

\subsection{General dimensions} Now we consider the case of general dimension $n\geq 1$. First, let us fix $j\in \{1,\dots, n\}$  and define the time-depending 
 interval $\mathcal{I}^{j}_{b}(t)$, given by
\begin{equation}\label{eq:Ijb}
\mathcal{I}_{b}^{j}(t)=\left\{ x\in \R^{n}~{} \vert {}~  |x_j|\geq (1+b)t \right\}, \quad t >2, \quad b>0.
\end{equation}

\begin{thm}\label{thm:outside}
Let $ m \in \R$ and $\psi\in C(\R; L^2(\R^n;\C^N))$ be a global solution to the Dirac equation \eqref{eq:dirac_1} with a potential $V$ such that
 \begin{equation}\label{eq:V}
    V(x,z)\in \R, \mbox{ with } (x,z)\in\R^n\times \C^N,
 \end{equation}
 and $b$ is an arbitrary positive number.
Then, on the region $\mathcal{I}_{b}^j(t)$ (see \eqref{eq:Ijb}), there is a strong decay to zero of the charge, i.e.  
\begin{equation}\label{decay0}
\lim_{t\to\infty}  \underset{~{}\mathcal{I}^j_{b}(t)}{\int} \psi^{\dagger}(t,x)\psi(t,x)dx
=0 \quad \mbox{ for all } j\in\{1,\dots,n\}.
\end{equation}
Moreover, there does not exist any non-decaying solution for the Dirac equation inside the region $\mathcal{I}^j_b(t)$, for any time $t$ sufficiently large.
\end{thm}

In other words, any global $L^2$ solution to the Dirac equation with (any) real potential $V$ must decay to zero outside these boxes, which are analogous to the exterior region of the multidimensional ``light cone''.
Moreover, the only hypothesis for the above results is the real value nature of the nonlinearity/potential.  This generality allows the results to be applied to various models within the Dirac equation family. Notably, the dynamics in the region $\mathcal{I}^j_{b}(t)$ are primarily governed by the behavior of the massless linear equation. 

Secondly, we observe that solitons (and other non-decaying solutions, such as breathers) move away from the region 
$\mathcal{I}^j_{b}(t)$ \eqref{eq:Ijb} in finite time. This implies that, at the $L^2$ level, non-decaying or periodic localized solutions outside the light cone are not possible for the nonlinear Dirac equation, both in the massless and massive cases, when the nonlinearity or potential 
$V$ is real and not necessarily symmetric. To the best of our knowledge, this is a novel result that highlights a unique property of global square-integrable solutions to the Dirac equation with real potential. For further details, see Section \ref{sec:appl}.

\medskip

As already mentioned in the 3D radial case, to prove Theorems \ref{thm:massless}, \ref{thm:massless2}, \ref{thm:3D} and \ref{thm:outside}, we shall introduce new virial techniques in the Dirac setting, in the spirit of previous works \cite{ACKM, MPS_ILW, munoz-kwak, MaMu, EM20}. Virial theorems are key elements which enables the conclusion that the dynamics in certain regions converge to zero when integrated over time. In the Dirac setting, this is not obvious from the nonpositive character of the operator. Indeed, finding positivity is one of new ingredients present in this paper. We shall not use the NLKG trick in this work, proving virial estimates directly on the Dirac's model. This is a relevant new contribution of this work to the study of decay in nonlinear Dirac's models.
 
\subsection{The 2D case} The 2D Dirac model has important applications in modern Physics. The classical Soler nonlinear Dirac equation reads
\[
\begin{aligned}
i\partial_t \psi_1 =&-i\partial_x \psi_2-\partial_y \psi_2 +(m-g(\overline{\psi}\psi))\psi_1\\
i\partial_t \psi_2 =&-i\partial_x \psi_1+\partial_y \psi_1 -(m-g(\overline{\psi}\psi))\psi_2.
\end{aligned}
\]
(see Section \ref{Pauli_matrices}.)  For $r>0$ and $\theta\in [0,2\pi)$, the polar form of the previous equation is given by
\[
\begin{aligned}
i\partial_t \psi_1 =&-e^{-i\theta}\left( i\partial_r+\frac{\partial_\theta}{r}\right) \psi_2 +(m-g(\overline{\psi}\psi))\psi_1\\
i\partial_t \psi_2 =&-e^{i\theta}\left( i\partial_r-\frac{\partial_\theta}{r}\right) \psi_1 +(m-g(\overline{\psi}\psi))\psi_2.
\end{aligned}
\]
Here, we can observe that the 2D case requires a subtle understanding, as the above equation is not purely radial, unlike the reduction presented in the 3D case.

The literature on 2D models is now extensive. It is known that, even in the massless regime, one has unexpected and interesting static solutions, such as Dirac bubbles and Killing spinors \cite{bubbles}. Here the authors gave a classification result for ground state solutions of the critical Dirac equation on $\R^n$ with $n\geq 2$. Moreover, such ground state solutions are also related to the Yamabe equation for the sphere.

Weinstein and Fefferman \cite{Honey2,Honey} studied the non-relativistic Schr\"odinger equation with a honeycomb lattice potential $V$, and concluded that for large time evolution for the nonlinear Schr\"odinger - Gross Pitaevskii equation for small amplitude initial conditions a 2D nonlinear Dirac system governs the dynamics. In this case,  the respective Dirac equation is given by

\begin{equation}\label{eq:Honey}
\begin{aligned}
i\partial_t \psi_1 =&-i\partial_x \psi_2-\partial_y \psi_2 +m\psi_1-(\beta_1|\psi_1|^2+\beta_2|\psi_2|^2)\psi_1\\
i\partial_t \psi_2 =&-i\partial_x \psi_1+\partial_y \psi_1 -m\psi_2+(\beta_2|\psi_1|^2+\beta_1 |\psi_2|^2)\psi_2,
\end{aligned}
\end{equation}
placed in $\mathbb R^2$ in a honeycomb, and $\beta_1,\beta_2$ are real valued positive parameters (see \cite{Honey,Transverse_Pelinovsky}). By length reasons, we decided to deal with \eqref{eq:Honey} and other 2D Dirac models elsewhere.

\subsection{Examples}\label{sec:appl}

To exemplify our result we will describe some nonlinearities and potentials such that Theorems \ref{thm:massless}, \ref{thm:massless2}, \ref{thm:3D} and \ref{thm:outside} hold. 

Recall Theorem \ref{thm:massless} in the massless case. Following \cite{CP_Block,Pelinovsky_Survey}, we have
\begin{enumerate}
\item $W=|u|^2|v|^2$ is the Thirring model;
\item $W=\frac12 (\overline{u}v+u\overline{v} )^2$ is the Gross-Neveu model;
\item $W=(|u|^2+|v|^2)|u|^2|v|^2$ appearing in the context of Feschback resonance for Bose-Einstein condensates.
\end{enumerate}
All these nonlinearities are covered by Theorem \ref{thm:massless}.  On the other hand, some classical real potentials $V$ is the case of 1D Thirring model.

\medskip

Now we consider nonlinearities in the massive case. It is known that, for instance, in the case $W=|u|^2|v|^2$ one has solitary waves. From Remark \ref{rem1p1} (we avoid the hats since they are not necessary) one has $u-v=i\psi_1$, $u+v=-\psi_2$,
\begin{equation}\label{Tirr}
\begin{aligned}
W_1 = & ~{} i(|u|^2 v -|v|^2 u)(T^{-1}\psi) = iuv(\bar u-\bar v)(T^{-1}\psi), \\
W_2 = & ~{} -(|u|^2 v + |v|^2 u)(T^{-1}\psi) = - uv(\bar u+ \bar v)(T^{-1}\psi).
\end{aligned}
\end{equation}
Therefore
\[
W_1 = \frac14 (\psi_1^2+\psi_2^2)\bar \psi_1, \quad W_2 = \frac14 (\psi_1^2+\psi_2^2)\bar \psi_2.
\]
Notice that $\partial_a W_1 -\partial_c W_2 = \frac12(|\psi_1|^2-|\psi_2|^2)\neq 0$, therefore Theorem \ref{thm:massless2} does not apply. 

On the other side,  notice that the following nonlinearities satisfy the condition \eqref{Laplacian} and parity condition 
 \[
\begin{aligned}
W_1=&~{} \psi_2(\psi_2^2-3\psi_1^2)-\overline{\psi_2}\big(\overline{\psi_2^2}-3\overline{\psi_1^2}\big),
\\
W_2=&~{} \psi_1(\psi_1^2-3\psi_2^2)+\overline{\psi_1}\big(\overline{\psi_1^2}-3\overline{\psi_2^2}\big),
\end{aligned}
\]
but do not satisfy \eqref{W_dependence}. Another example is given as follows: for  $m, n\in \bb{N}$ odd and  fixed and  $a_1,a_2,b_1,b_2\in \C\setminus\{0\}$ such that $a_1^2+a_2^2=0$ and $b_1^2+b_2^2=0$, consider
\[
\begin{aligned}
W_1=&~{}-\frac{a_1}{a_2}(a_1 \psi_1 +a_2 \psi_2)^m+\frac{b_1}{b_2}(b_1\overline{\psi_1}+b_2\overline{\psi_2})^n,
\\
W_2=&~{}(a_1 \psi_1 +a_2 \psi_2)^m+(b_1\overline{\psi_1} +b_2\overline{\psi_2})^n.
\end{aligned}
\]
In this case we have that the parity condition holds, but \eqref{W_dependence} is not satisfied. 
Finally, let us consider
\[
W=\overline{\psi_1^4}+\overline{\psi_2^4}-6\overline{\psi_1^2}\overline{\psi_2^2}.
\]
Therefore, one gets
\[
W_1=\partial_{\overline{\psi}_1}W=4\overline{\psi_1^3}-12\overline{\psi_1}\overline{\psi_2^2},\quad
W_2=\partial_{\overline{\psi}_2}W=4\overline{\psi_2^3}-12\overline{\psi_1^2}\overline{\psi_2}
\]
which satisfy the conditions \eqref{Laplacian} and \eqref{W_dependence}, and the parity condition. Therefore Theorem \ref{thm:massless2} applies in this case.

{
Now, we focus on the 3D case in \eqref{eq:PW3}.  Let us consider $A\in \mathcal{M}_{4\times 4}(\R)$ and $X=(\Re \phi, \Im\phi)$. For $A=\mbox{diag}(a_{11},a_{21},a_{12},a_{22})$, we set
\begin{equation}\label{W_3d}
(W_1,W_2):=g (X^{T}AX) X,
\end{equation}
with $g:\R\to\R$ a given polynomial nonlinearity of order at least 1, satisfying $g(0)=0$. Thus, it holds that $|g(s)|\lesssim s^{p}$, with $p\geq1$, $p\in \bb{N}$.  Then, we obtain
\begin{equation*}
W_{jk} = g\left( \sum_{l,m=1,2} a_{lm}\phi_{lm}^2\right)  \phi_{jk}, 
\quad  \mbox{ for } j,k=1,2.  
\end{equation*}
This type of nonlinearity satisfies \eqref{eq:WW_w}. Therefore, this construction provides a class of nonlinearities for which Theorem \ref{thm:3D} holds. Notice that Soler's nonlinearity \eqref{Soler} fits into \eqref{W_3d}, and therefore satisfies the hypotheses of Theorem \ref{thm:3D}.
}

\subsubsection{Cubic Dirac equation with Hartree-type nonlinearity in $\R^{1+3}$.}  Now we exemplify Theorem \ref{thm:outside}. Here, we will consider a simplified version of the Dirac-Klein-Gordon equation (see \cite{S22,Seokchang,GS_hartree} and the reference therein). Consider 
\begin{equation}\label{eq:dirac_potential_KG}
\begin{aligned}
i\partial_{t} \psi + i\bal \cdot\nabla \psi-m\beta \psi=&~{}\phi \beta \psi, \quad (t,x)\in \R\times\R^{n},\\
(\partial_t^2-\Delta+M^2)\phi =&-\psi^{\dagger}\psi.
\end{aligned}
\end{equation}
An example of classical solution is the following. Assume that the pseudo scalar field $\phi$ is an standing wave, given by $\phi(t,x)=e^{i\lambda t}\vA (x)$, with $M>\lambda$. Then, the equation related to the Klein-Gordon equation becomes
\[
(-\Delta+M^2-\lambda^2)\phi =-\psi^{\dagger}\psi.
\]
Rewriting the equation \eqref{eq:dirac_potential_KG}, one obtains
\begin{equation*}
\begin{aligned}
i\partial_{t} \psi + i\bal \cdot\nabla \psi-m\beta \psi=&~{}-[V_b*(\psi^{\dagger}\psi)](x) \beta \psi, 
\\ V_b(x)=&~{} \frac{e^{-b|x|}}{|x|}, \qquad  (t,x)\in \R\times\R^{n}.
\end{aligned}
\end{equation*}

\begin{cor}
Let $(\psi,\phi)$ be a global solution to \eqref{eq:dirac_potential_KG}, such that $\psi\in L^2$. Then, Theorem \ref{thm:outside} holds in this case. 
\end{cor}

\subsection*{Organization of this paper}	This paper is organized as follows: Section \ref{sec:virial} deals with a new virial identity introduced in this paper. Section \ref{thm:outside_proof} is devoted to the proof of Theorem \ref{thm:outside}. Section \ref{sec:thm1p1} deals with the proof of Theorem \ref{thm:massless} in the massless case, and Section \ref{sec:thm1p2} considers the massive case (Theorem \ref{thm:massless2}). Finally, in Section \ref{Section:3D} we prove Theorem \ref{thm:3D}.

\subsection*{Acknowledgements} We thank Hanne Van Den Bosch, Dmitry Pelinovsky and Luca Fanelli for constructive and stimulating comments and suggestions on a first version of this paper.
\section{Preliminares}

\subsection{Dirac and Pauli Matrices}\label{Pauli_matrices}

The Dirac operator $\ca{D}_m$ is described using the Dirac matrices, as well as, the $\ca{H}$ Hamiltonian part \eqref{H}. 
$\gamma$-matrices are given by
\[\begin{aligned}
\gamma^{0}=\beta=\left(\begin{matrix}
I_{N/2}&0\\
0& -I_{N/2}
\end{matrix}\right), 
\quad
\gamma^j=\left(\begin{matrix}
0 &\sigma^j\\
-\sigma^j & 0
\end{matrix}\right), 
\end{aligned}
\]
associated with Pauli matrices given by
\[
\begin{aligned}
 \sigma^1=\left(\begin{matrix}
0 &1\\
1&0
\end{matrix}\right),  
\quad
\sigma^2=\left(\begin{matrix}
0 &-i\\
i&0
\end{matrix}\right),
\quad
\sigma^3=\left(\begin{matrix}
1 &0\\
0&-1
\end{matrix}\right).
\end{aligned}
\]

And recalling that the Hamiltonian $\ca{H}$ is given by 
\begin{equation}\label{eq:Op_Dm}
\begin{aligned}
\mathcal{H} = &~{} -i\bal\cdot\nabla, \quad \bal=(\al^1,\al^2, \cdots ,\al^n), \mbox{ and}
\\
\alpha^j= &~{} 
\beta \gamma^j=
\left(
\begin{matrix}
0&\sigma^j\\
\sigma^j&0
\end{matrix}
\right).
\end{aligned}
\end{equation}
with $\al^k$, $k\in\{1,\cdots, n\}$, and $\beta$ being self-adjoint matrices satisfying the relations
\begin{equation}\label{eq:alpha_matrix}
\begin{aligned}
\al^j \al^k+\al^k \al^j=2\delta_{jk} I_{N},\quad & \quad \al^j \beta +\beta \al^j=0, \\
\quad (\al^j)^{\dagger}=\al^j,~{} (\al^j)^2=\beta^2&=I_N,
\end{aligned}
\end{equation}
for  $j,k\in \{1,\cdots,n\}$ (see \cite{Thaller}.)  The following are the most classical choices in the 1D, 2D, and 3D cases:
\begin{enumerate}
\item {One dimensional case}: We have  $\al^1=-\sigma^2$ and $\beta=\sigma^3$.

\smallskip

\item {Two dimensional case}: In this case, we consider $\al^1=\sigma^1$, $\al^2=\sigma^2$ and $\beta=\sigma^3$. 

\smallskip

\item {Three dimensional case:} Here $\al^j=\left( \begin{matrix} 0& \sigma^j\\ \sigma^j&0\end{matrix}\right)$.

\end{enumerate}

For more details about the extension of the Pauli matrices to higher dimensions, see \cite{OY04}.

\subsubsection{Rewriting of Dirac system}\label{ss:Rewri}
In this section, it is gathered all the necessary auxiliary results which will be needed in forthcoming sections. We start by presenting a suitable decomposition of the Dirac equation, which will be useful in the main Theorem's proof. In what follows, we will consider a suitable decomposition 
of the equation \eqref{eq:dirac_1}. This will be useful in some intermediate steps on the computation of the virial identity (see \eqref{eq:I4_u1u2}).
Let $\psi$ a solution to \eqref{eq:dirac_1} such that
\[
\psi=\bua+i\bub \mbox{ with }\bua,\bub\in \bb{R}^N ,
\]
and
\[
\bal=\Re \bal + i \Im \bal=:\bal_r+i\bal_i.
\]
where $\bal$ is given in \eqref{eq:Op_Dm}.
\medskip

Then, for the sake of completeness, by rewriting the system in these variables, we get 
\[
\begin{aligned}
\partial_t \left(\begin{matrix} \bua\\ \bub \end{matrix}\right) 
=  -\left(\begin{matrix} \bal_r  \cdot \nabla&-\bal_i  \cdot \nabla \\ \bal_i  \cdot \nabla & \bal_r  \cdot \nabla\end{matrix}\right) \left(\begin{matrix} \bua\\ \bub \end{matrix}\right) 
+(m+V(\psi))  \left(\begin{matrix} 0& \beta \\-\beta &0\end{matrix}\right) \left(\begin{matrix}  \bua\\ \bub \end{matrix}\right) 
\end{aligned}
\]
or equivalently
\[
\begin{aligned}
\partial_t \bua
=&- (\bal_r  \cdot \nabla \bua  -\bal_i  \cdot \nabla \bub) +(m+V(\psi)) \beta \bub
\\
 \partial_t \bub
=&-(\bal_i  \cdot \nabla\bua +\bal_r  \cdot \nabla\bub)-(m+V( \psi)) \beta \bua.
\end{aligned}
\]
In what follows, we will not use this representation but it is included by completeness. 
Notice that from \eqref{eq:alpha_matrix}, and assuming that $\alpha^j$ is such that
\begin{equation}\label{eq:alpha_ri}
\alpha^j=\Re \alpha^j\mbox{ or }\alpha^j =i \Im \alpha^j,
\end{equation}
the conditions on \eqref{eq:alpha_matrix} are equivalent to
\begin{equation}\label{rem:alp_r_i}
\begin{aligned}
 \al_r^j\al_r^k+\al_r^k\al_r^j-(\al_i^j\al_i^k &+\al_i^k\al_i^j)=2\delta_{jk}I_N, \\
\quad (\al_r^j)^2-(\al_i^j)^2=I_N,\quad& \al_r^j\al_i^k+\al_i^k\al_r^j=0,\;\quad 
\end{aligned}
\end{equation}
for all $  j,k\in \{1,\cdots,n\}.$ Moreover, the conserved quantities \eqref{eq:energy} and \eqref{eq:charge} in terms of the new variables $\bua$ and $\bub$ are given by
\begin{equation*}
Q(\psi)
= \int (\bua^{\dagger} \bua+ \bub^{\dagger} \bub),
\end{equation*}
and
\begin{equation*}
\begin{aligned}
E(\psi)
=& \int \left[ \bua^{\dagger} \bal_i \cdot \nabla \bua+\bub^{\dagger} \bal_i \cdot \nabla  \bub +2 \bua^{\dagger} \bal_r \cdot \nabla  \bub\right]
\\&+\int \left[ m(\bua^{\dagger} \beta\bua+\bub^{\dagger} \beta\bub)- G(\bua^{\dagger} \beta\bua+\bub^{\dagger} \beta\bub)\right].
\end{aligned}
\end{equation*}
for $V=-g(\overline{\psi}\psi)$  and $G'=g$.

\medskip

Motivated by the above decomposition, we shall deal with the following operators $\bal_r\cdot \nabla$ and $\bal_i\cdot \nabla$. The next lemma will be useful in handling these operators besides establishing the virial identities. In order to prove these identities, we shall need the following identity.

\begin{lem}\label{lem:alp_IBP}
Let $\vA $ and smooth function, and  $\bbf$ and $\bbg$ in $H^1(\R^n)$, then the following identities holds
\begin{equation}\label{eq:f_alr_dg}
\begin{aligned}
\int \vA  \bbf^{\dagger} \bal_r\cdot \nabla \bbg=- \sum_{j=1}^{n} \int  \partial_{j} \vA     \bbf^{\dagger}\al_{r}^{j} \bbg 
  	-  \int \vA     \bbg^{\dagger}\bal_{r}\cdot \nabla \bbf,
\end{aligned}
\end{equation}
and 
\begin{equation}\label{eq:f_ali_dg} 
\begin{aligned}
\int \vA   \bbf^{\dagger} \bal_i\cdot \nabla \bbg= -\sum_{j=1}^{n} \int  \partial_{j} \vA     \bbf^{\dagger}\al_{i}^{j} \bbg 
  	+\int \vA     \bbg^{\dagger}\bal_{i}\cdot \nabla  \bbf .
\end{aligned}
\end{equation}
\end{lem}

\begin{proof}

Firstly, we notice that from \eqref{eq:alpha_matrix},  \eqref{eq:alpha_ri} and \eqref{rem:alp_r_i}, we get
\begin{equation}\label{eq:dagger_alp}
(\al_r^j)^{\dagger}=\al_r^j \quad \mbox{ and } \quad(\al_i^j)^{\dagger}=-\al_i^j.
\end{equation}
Let $\bal_p$ with $p\in \{r,i\}$, and consider
\begin{equation}\label{eq:auxiliar}
 \int  \vA    \bbf^{\dagger} \bal_p \cdot\nabla \bbg.
\end{equation}

Rewriting \eqref{eq:auxiliar} and integrating by parts, one gets
\[
\begin{aligned}
 \int  \vA    \bbf^{\dagger} \bal_p \cdot\nabla \bbg
 =&\sum_{j=1}^{n} \int  \vA     \bbf^{\dagger} \al_p^{j} \partial_{j} \bbg
  =-\sum_{j=1}^{n} \int \partial_{j}  \left(\vA      \bbf^{\dagger}\al_{p}^{j} \right) \bbg  \\
  =&-\sum_{j=1}^{n} \int  \left( \partial_{j} \vA     \bbf^{\dagger}\al_{p}^{j}+\vA      \partial_{j}\bbf^{\dagger}\al_{p}^{j} \right) \bbg .   
\end{aligned}
\]
It follows that
\[
\begin{aligned}
 \int  \vA    \bbf^{\dagger} \bal_p \cdot\nabla \bbg
  =&-\sum_{j=1}^{n} \int  \partial_{j} \vA     \bbf^{\dagger}\al_{p}^{j} \bbg 
  	-\sum_{j=1}^{n} \int \vA      \partial_{j}\bbf^{\dagger}\al_{p}^{j}  \bbg .
\\
  =&-\sum_{j=1}^{n} \int  \partial_{j} \vA     \bbf^{\dagger}\al_{p}^{j} \bbg 
  	-\sum_{j=1}^{n} \int \vA     \bbg^{\dagger}(\al_{p}^{j})^{\dagger}  \partial_{j} \bbf .  
\end{aligned}
\]
For $p=r$, using \eqref{eq:dagger_alp}, we get
\[
\begin{aligned}
 \int  \vA    \bbf^{\dagger} \bal_r \cdot\nabla \bbg
  =&-\sum_{j=1}^{n} \int  \partial_{j} \vA     \bbf^{\dagger}\al_{r}^{j} \bbg 
  	-  \int \vA     \bbg^{\dagger}\bal_{r}\cdot \nabla \bbf .
\end{aligned}
\]
This proves \eqref{eq:f_alr_dg}. Finally, for $p=i$ and using  \eqref{eq:dagger_alp}, one has
\[
\begin{aligned}
 \int  \vA    \bbf^{\dagger} \bal_i \cdot\nabla \bbg
  =&-\sum_{j=1}^{n} \int  \partial_{j} \vA     \bbf^{\dagger}\al_{i}^{j} \bbg 
  	+\sum_{j=1}^{n} \int \vA     \bbg^{\dagger}\al_{i}^{j}  \partial_{j} \bbf 
	\\
 =&-\sum_{j=1}^{n} \int  \partial_{j} \vA     \bbf^{\dagger}\al_{i}^{j} \bbg 
  	+\int \vA     \bbg^{\dagger}\bal_{i}\cdot \nabla  \bbf .  
\end{aligned}
\]
This concludes the proof of the lemma.
\end{proof}

Moreover, as a direct consequence, we obtain the following result.

\begin{cor}\label{cor:f_ar_df}
Let $\vA $ an smooth function and $\bbf$ in $L^2(\R^n)$. The following identities hold
\[
\begin{aligned}
\int \vA  \bbf^{\dagger} \bal_r\cdot \nabla \bbf= -\frac12\sum_{j=1}^{n} \int  \partial_{j} \vA     \bbf^{\dagger}\al_{r}^{j} \bbf .
\end{aligned}
\]
\end{cor}

\section{Virial Identities}\label{sec:virial}

\subsection{Identities for the $n$ dimensional Dirac equation}

Let $\vA :\R^{n}\to \R$ a bounded smooth weight function. Let $\I$ be a modified version of the charge \eqref{eq:charge}, i.e. for $\psi$ a solution of the IVP \eqref{eq:dirac_1}, we define the functional
\[
\I= \frac{1}{\mu(t)} \int \vA \left(\frac{x+\rho(t)}{\lambda(t)}\right) \psi^{\dagger}(t,x)\psi(t,x)dx.
\]
 For $\psi=\psi(t,\cdot)\in C(\R;L^2(\R^n;\C^N))$ and considering $1/\mu$ is bounded and positive,  one obtain that the functional $\I$ is well-defined and 
\[
\sup_{t\in \R}\I(t)<\infty
\]
Now, we present one of the main virial identity of this work.
\begin{prop}\label{prop:dtI}
Let $\vA $ an smooth bounded function and $\psi$ a nontrivial solution of the IVP \eqref{eq:dirac_1} such that 
\[
\psi=\bua+i\bub \in C^1_{loc}(\R: L^2(\R^n;\C^N))
\]  
Then the following identity holds
\begin{equation}\label{eq:dtI}
\begin{aligned}
\dt \I
=& 
-\frac{\mu'(t)}{\mu^{2}(t)} \int \vA \left(\frac{x+\rho(t)}{\lambda(t)}\right) \left[\bua^{\dagger} \bua+\bub^{\dagger} \bub\right]
\\&
+\frac{1}{\mu(t)} \frac{1}{\lambda(t)}  \sum_{j}  \int \partial_{j} \vA \left(\frac{x+\rho(t)}{\lambda(t)}\right)  \rho'_{j}(t) 
\left[\bua^{\dagger} \bua+\bub^{\dagger} \bub\right]
\\&-\frac{1}{\mu(t)}  \frac{\dot \lambda(t) }{\lambda(t)} \sum_{j}  \int \partial_{j} \vA \left(\frac{x+\rho(t)}{\lambda(t)}\right)  \frac{(x_{j}+\rho_{j}(t))}{\lambda(t)}
\left[\bua^{\dagger} \bua+\bub^{\dagger} \bub\right]
\\&
- \frac{1}{2\mu(t) \lambda(t)} \sum_{j=1}^{n} \int \partial_{j} \vA \left(\frac{x+\rho(t)}{\lambda(t)}\right)  [   
																			\bua^{\dagger}\al_{r}^{j}  \bua
 																			+\bub^{\dagger}\al_{r}^{j}  \bub
																			-2 \bua^{\dagger} \al_i^{j}  \bub
																			]
																			.
\end{aligned}
\end{equation}
\end{prop}
\begin{proof}
Taking the derivative with respect to time on $\I$, we obtain
\[
\begin{aligned}
\dt \I
=& 
-\frac{\mu'(t)}{\mu^{2}(t)} \int \vA \left(\frac{x+\rho(t)}{\lambda(t)}\right) \psi^{\dagger}(t,x)\psi(t,x)\\
&+\frac{1}{\mu(t)} \int  \sum \partial_{i} \vA \left(\frac{x+\rho(t)}{\lambda(t)}\right)\partial_{t}\left(\frac{x_{i}+\rho_{i}(t)}{\lambda(t)}\right) \psi^{\dagger}(t,x)\psi(t,x)
\\&
 +\frac{2}{\mu(t)} \Re  \int \vA \left(\frac{x+\rho(t)}{\lambda(t)}\right) \psi^{\dagger}(t,x)\partial_{t} \psi(t,x).
 \end{aligned}
 \]
 A further simplification gives 
 \begin{equation}\label{eq:I_intermedio}
\begin{aligned}
\dt \I  =&
-\frac{\mu'(t)}{\mu^{2}(t)} \int \vA \left(\frac{x+\rho(t)}{\lambda(t)}\right) \psi^{\dagger}(t,x)\psi(t,x)
\\&
+\frac{1}{\mu(t)}   \sum_{i} \frac{\rho'_{i}(t)}{\lambda(t)} \int \partial_{i} \vA \left(\frac{x+\rho(t)}{\lambda(t)}\right)
\psi^{\dagger}(t,x)\psi(t,x)
\\&
- \frac{\dot\lambda(t)}{\lambda(t)}\frac{1}{\mu(t)} \sum_{i}\int \partial_{i} \vA \left(\frac{x+\rho(t)}{\lambda(t)}\right) \frac{(x_{i}+\rho_{i}(t))}{\lambda(t)}
\psi^{\dagger}(t,x)\psi(t,x)
\\&
 +\frac{2}{\mu(t)} \Re  \int \vA \left(\frac{x+\rho(t)}{\lambda(t)}\right) \psi^{\dagger}(t,x)\partial_{t} \psi(t,x)
 \\
 =:~{}& I_{1}+I_{2}+I_{3}+I_{4}
 .
\end{aligned}
\end{equation}
We will focus on $I_{4}$. Replacing  \eqref{eq:dirac_1}, one gets
\[
\begin{aligned}
I_4
=&- \frac{2}{\mu(t)} \Re  \int \vA \left(\frac{x+\rho(t)}{\lambda(t)}\right) \psi^{\dagger}(i \mathcal{H}\psi+i(m+V)\beta\psi)\\
=&- \frac{2}{\mu(t)} \Re  \int \vA \left(\frac{x+\rho(t)}{\lambda(t)}\right) \psi^{\dagger}i \mathcal{H}\psi
\\&
- 2\frac{1}{\mu(t)} \Re  \int i \vA \left(\frac{x+\rho(t)}{\lambda(t)}\right) (m+V) \psi^{\dagger}\beta\psi.
\end{aligned}
\]
Since $(m+V )\psi^{\dagger}\beta \psi$ is a real value (see \eqref{eq:V}), the last term on the above equality is zero. Now, replacing operator $\mathcal{H}$ (see \eqref{eq:Op_Dm}), we obtain
\[
\begin{aligned}
I_4
=&- \frac{2}{\mu(t)} \Re  \int \vA \left(\frac{x+\rho(t)}{\lambda(t)}\right) \psi^{\dagger}\bal\cdot \nabla \psi.
\end{aligned}
\]
Considering $\psi=\bua+i\bub$ and $\bal=\bal_r+i\bal_i$ as in Section \ref{ss:Rewri}, we can rewrite the above term as
\[
\begin{aligned}
 I_4
=&-\frac{2}{\mu(t)} \Re  \int \vA    (\bua+i\bub)^{\dagger} (\bal_r+i\bal_i)\cdot\nabla(\bua+i\bub)\\
=&-\frac{2}{\mu(t)} \Re  \int \vA  
				\left[
				\bua^{\dagger} (\bal_r \cdot\nabla \bua -\bal_i \cdot\nabla\bub)
				+\bub^{\dagger} (\bal_r \cdot\nabla \bub+\bal_i \cdot\nabla \bua)
				\right]\\
&-\frac{2}{\mu(t)} \Re   \int \vA    
				i \left[
				 \bua^{\dagger} (\bal_r \cdot\nabla \bub+\bal_i \cdot\nabla \bua)
				-\bub^{\dagger}(\bal_r \cdot\nabla \bua -\bal_i \cdot\nabla\bub)
				\right].
\end{aligned}
\]
Here, $\vA  =\vA  \left(\frac{x+\rho(t)}{\lambda(t)}\right)$. Then, we obtain that 
\begin{equation}\label{eq:I4_u1u2}
\begin{aligned}
I_{4}=&
-\frac{2}{\mu(t)}   \int \vA \left(\frac{x+\rho(t)}{\lambda(t)}\right)  
				\left[
				\bua^{\dagger} \bal_r \cdot\nabla \bua 
				+ \bub^{\dagger} \bal_r \cdot\nabla \bub
				\right]\\
&
-\frac{2}{\mu(t)}   \int \vA \left(\frac{x+\rho(t)}{\lambda(t)}\right)  
				\left[
				-\bua^{\dagger} \bal_i \cdot\nabla\bub
				+ \bub^{\dagger} \bal_i \cdot\nabla \bua
				\right].
\end{aligned}
\end{equation}
Now, we focus on the first integral on the RHS of \eqref{eq:I4_u1u2}. 

Applying Corollary \ref{cor:f_ar_df}, we get
 \begin{equation}\label{eq:I41}
\begin{aligned}
& \int  \vA \left(\frac{x+\rho(t)}{\lambda(t)}\right)    \left(\bua^{\dagger} \bal_r \cdot\nabla \bua  + \bub^{\dagger} \bal_r \cdot\nabla \bub\right)\\
& \quad =- \frac{1}{2\lambda(t)} \sum_{j=1}^{n} \int \partial_{j} \vA \left(\frac{x+\rho(t)}{\lambda(t)}\right)   \left(  \bua^{\dagger}\al_{r}^{j}  \bua+ \bub^{\dagger}\al_{r}^{j}  \bub \right).
\end{aligned}
\end{equation}
Now, we focus on the second integral of \eqref{eq:I4_u1u2}. Similar as before, by \eqref{eq:f_ali_dg} in Lemma \ref{lem:alp_IBP}, one gets
\begin{equation}\label{eq:I42}
\begin{aligned}
 & \int \vA \left(\frac{x+\rho(t)}{\lambda(t)}\right)  \left(\bua^{\dagger} \bal_i \cdot\nabla\bub- \bub^{\dagger} \bal_i\cdot \nabla  \bua \right)  
\\
 & = - \frac{1}{\lambda(t)}\sum_{j=1}^{n}  \int \partial_{j} \vA \left(\frac{x+\rho(t)}{\lambda(t)}\right)  \bua^{\dagger} \al_i^{j}  \bub.
\end{aligned}
\end{equation}
Gathering $I_1$, $I_2$, $I_3$ in \eqref{eq:I_intermedio},  \eqref{eq:I41} and  \eqref{eq:I42}; and replacing $\psi=\bua+i\bub$, we obtain
\[
\begin{aligned}
\dt \I
 =& 
-\frac{\mu'(t)}{\mu^{2}(t)} \int \vA \left(\frac{x+\rho(t)}{\lambda(t)}\right) \left[\bua^{\dagger} \bua+\bub^{\dagger} \bub\right]
\\&
+\frac{1}{\mu(t)}   \sum_{j} \frac{\rho'_{j}(t)}{\lambda(t)} \int \partial_{j} \vA \left(\frac{x+\rho(t)}{\lambda(t)}\right)
\left[\bua^{\dagger} \bua+\bub^{\dagger} \bub\right]
\\&
- \frac{1}{\mu(t)}\frac{\dot \lambda(t)}{\lambda(t)} \sum_{j}\int \partial_{j} \vA \left(\frac{x+\rho(t)}{\lambda(t)}\right) \frac{(x_{j}+\rho_{j}(t))}{\lambda(t)}
\left[\bua^{\dagger} \bua+\bub^{\dagger} \bub\right]
\\&
- \frac{1}{\mu(t)}\frac{1}{ \lambda(t)} \sum_{j=1}^{n} \int \partial_{j} \vA \left(\frac{x+\rho(t)}{\lambda(t)}\right)  [   
																			\bua^{\dagger}\al_{r}^{j}  \bua
 																			+\bub^{\dagger}\al_{r}^{j}  \bub
																			- 2\bua^{\dagger} \al_i^{j}  \bub
																			].																
\end{aligned}
\]
Regrouping the terms, we get
\begin{equation*}
\begin{aligned}
 \dt \I
=& 
-\frac{\mu'(t)}{\mu^{2}(t)} \int \vA \left(\frac{x+\rho(t)}{\lambda(t)}\right) \left[\bua^{\dagger} \bua+\bub^{\dagger} \bub\right]
\\&
+\frac{1}{\mu(t)} \frac{1}{\lambda(t)}  \sum_{j}  \int \partial_{j} \vA \left(\frac{x+\rho(t)}{\lambda(t)}\right) \left[ \rho'_{j}(t) - \dot\lambda(t) \frac{(x_{j}+\rho_{j}(t))}{\lambda(t)}
 \right]
\left[\bua^{\dagger} \bua+\bub^{\dagger} \bub\right]
\\&
- \frac{1}{2\mu(t)}\frac{1}{\lambda(t)} \sum_{j=1}^{n} \int \partial_{j} \vA \left(\frac{x+\rho(t)}{\lambda(t)}\right)  [   
																			\bua^{\dagger}\al_{r}^{j}  \bua
 																			+\bub^{\dagger}\al_{r}^{j}  \bub
																			-2 \bua^{\dagger} \al_i^{j}  \bub
																			]
.
\end{aligned}
\end{equation*}
This concludes the proof.
\end{proof}

\begin{rem}
One remarkable property of the above virial identity is that the dynamics is dominated by the massless linear operator. This observation suggests that, in certain regions, the behavior of the massive Dirac equation closely behaves the massless Dirac equation. Furthermore, the (real) nonlinearity appears to play no role in these dynamics, which aligns with Soler's findings in \cite{soler}.  
In that work, the author showed that only a very small fraction of the solitary wave's energy in the three-dimensional case is attributable to self-interaction.
\end{rem}

\begin{rem}
Notice that the virial identity is independent of the mass and the real-valued nonlinearity. This independence is a notable feature, allowing us to apply Theorem \ref{thm:outside} to a wide range of models, such as the Soler model, the Thirring model, or the Dirac equation with Hartree-type nonlinearities, among others (see Section \ref{sec:appl}).
\end{rem}

 \subsection{1D Virial identity}
 
Now we consider as a particular case the 1D model. Let us consider $\phi$ be a real-valued, smooth, and bounded function.  Consider the functionals
 \[
 \begin{aligned}
 \mathcal{K}=\int  \phi \left(\frac{x}{\lambda(t)}\right) \big( |u|^2+|v|^2\big)
 ~{}
 \mbox{ and }~{}
  \mathcal{J}=\int \phi\left(\frac{x}{\lambda(t)}\right)( |u|^2-|v|^2) .
 \end{aligned}
 \]
 For all bounded function $\phi$, the functionals $ \mathcal{J}$ and $ \mathcal{K}$ are well defined and it follows that
\[
\sup_{t\in \R} \left( |\J| + |\mathcal K(t)| \right)  <\infty.
\]
We claim the following result.

 \begin{lem}\label{cor:dtI1d}
Let $m\in \R$ and  $(u,v)\in (L^2\times L^2)(\mathbb R; \C)$ be a global solution to \eqref{eq:D_LC}. Then
 \begin{equation}\label{eq:dtI1d}
 \dt \mathcal{K}=     \int \phi' (|u|^2-   |v|^2 ).
 \end{equation}
 \end{lem}
 \begin{proof}
 Taking derivative on time and replacing \eqref{eq:D_LC}, one gets
  \[
  \begin{aligned}
 \dt \mathcal{K} = 2\Re  \int \phi( u\dot{\overline{u}}+v\dot{\overline{v}})
  = &~{} -2\Im  \int \phi u \overline{(- i\partial_x u -mv+\partial_{\overline{u}} W(u,v))}
  \\&
-2 \Im  \int \phi v \overline{(i  \partial_xv-mu+\partial_{\overline{v}} W(u,v))}  .
\end{aligned}
 \]
 Rearranging the terms, we obtain
   \[
  \begin{aligned}
 \dt \mathcal{K}
    =&~{} -2\Re  \int \phi (u \partial_x \overline{u}-   v \partial_x\overline{v} )
   +2m\Im  \int \phi (u \overline{v}+v \overline{u} )
   \\&
  - 2\Im  \int \phi \big( u \overline{\partial_{\overline{u}} W(u,v)}+v \overline{\partial_{\overline{v}} W(u,v))} \big).
  \end{aligned}
 \]
Finally, using \eqref{en:1} and \eqref{en:3}, $( u \overline{\partial_{\overline{u}} W(u,v)}+v\overline{\partial_{\overline{v}} W(u,v)}) \in \R $, $u \overline{v}+v \overline{u} =2\Re(u \overline{v})$, and integrating by parts, we obtain
   \[
  \begin{aligned}
 \dt \mathcal{K}
      =&~{} -2\Re  \int \phi (u \partial_x \overline{u}-   v \partial_x\overline{v} )
   +4m\Im  \int \phi \Re( u \overline{v})
   \\&
  - 2\Im  \int \phi \big( u \overline{\partial_{\overline{u}} W(u,v)}+v \overline{\partial_{\overline{v}} W(u,v))} \big)
    \\ = &~{}  \frac1{\lambda(t)} \int \phi' \left(\frac{x}{\lambda(t)}\right) (|u|^2-   |v|^2 ).
\end{aligned}
 \]
This concludes the proof of lemma.
 \end{proof}
 
  \begin{lem}\label{lem:u-v}
  Let $m\in \R$ and  $(u,v)\in (L^2\times L^2)(\R,\C)$ a global solution to \eqref{eq:D_LC}.
 \begin{equation}\label{eq:dt_J}
 \begin{aligned}
 \dt \J
 =&~{} 
- \frac{\dot\lambda}{\lambda}\int \frac{x}{\lambda} \phi' \left(\frac{x}{\lambda}\right)( |u|^2-|v|^2)
  \\
  &~{} -\frac{1}{2\lambda} \int \phi' \left(\frac{x}{\lambda}\right)( |u|^2+ |v|^2  )
   -4m\Im  \int \phi \left(\frac{x}{\lambda}\right)u \overline{v}.
 \end{aligned}\end{equation}
 \end{lem}
 
 Notice the difference between massless and massive cases. In the former, \eqref{eq:dt_J} gives zero if $\phi$ is constant, while in the latter this is not the case.
 \begin{proof}
 Derivating with respect to time, we get (dot denotes time derivative)
  \[
  \begin{aligned}
 \dt \J =&~{}
- \frac{\dot\lambda}{\lambda}\int \frac{x}{\lambda} \phi'\left(\frac{x}{\lambda}\right)( |u|^2-|v|^2)
 + 2\Re  \int \phi( u\dot{\overline{u}}-v\dot{\overline{v}})=: J_1+ J_2.
     \end{aligned}
 \]
 Now, we focus on $J_2$. We observe that
   \[
  \begin{aligned}
\frac12 J_2
  =&~{} -\Im  \int \phi( u\dot{\overline{iu}}-v\dot{\overline{iv}}).
    \end{aligned}
 \]
 Using \eqref{eq:D_LC} and rearranging the terms, we obtain
   \[
  \begin{aligned}
\frac{1}{2} J_2
  =&~{} -\Re  \int \phi (u \partial_x \overline{u}+   v \partial_x\overline{v} )
  \\
  &~{} +m\Im  \int \phi (u \overline{v}-v \overline{u} )
  + \Im  \int \phi \big( u \overline{\partial_{\overline{u}} W(u,v)}-v \overline{\partial_{\overline{v}} W(u,v)} \big).
  \end{aligned}
   \]
Since $ u \overline{\partial_{\overline{u}} W(u,v)}-v \overline{\partial_{\overline{v}} W(u,v))}\in \R$, and $u \overline{v}-v \overline{u}=2i\Im(u \overline{v})$, we obtain
   \[
  \begin{aligned}
\frac{1}{2} J_2
    =&~{} -\Re  \int \phi (u \partial_x \overline{u}+   v \partial_x\overline{v} )
   +m\Im  \int \phi (u \overline{v}-v \overline{u} ) \\
   =&~{}\frac12  \int \phi' ( |u|^2+ |v|^2  )
   +2m\Im  \int \phi u \overline{v}.
\end{aligned}
 \]
Finally, gathering $J_1$ and $J_2$, 
\[
\begin{aligned}
\dt \J=&~{}
- \frac{\dot\lambda}{\lambda }\int \frac{x}{\lambda} \phi'\left(\frac{x}{\lambda}\right)( |u|^2-|v|^2)\\
&~{} 
+\frac{1}{\lambda}  \int \phi' \left(\frac{x}{\lambda}\right)( |u|^2+ |v|^2  )
   +4m \Im  \int \phi \left(\frac{x}{\lambda}\right) u \overline{v}.
\end{aligned}
\]
\end{proof}
This last identity will be useful for the proof of Theorem \ref{thm:massless} in the massless case.

\section{Proof of Theorem \ref{thm:outside}}\label{thm:outside_proof}

In this section we prove the decay of global solutions in regions of space characterized by fast speed.

\subsection{First computations}

Let $\lambda>0$ a real constant, and
\[
\mu(t)=1,\qquad \lambda(t)=\lambda,\qquad \rho(t)= t{\bbb \theta}, 
\]
where  ${\bbb \theta}=(\theta_1,\cdots, \theta_n)$ is a fixed vector on $\R^n$. Then, by Proposition \ref{prop:dtI}, we get
\medskip
\[
\begin{aligned}
\dt \I 
\leq&~{} 
 \frac{1}{\lambda}   \sum_{j=1}^n \int \left( \partial_{j} \vA \left(\frac{x+{\bbb \theta}t}{\lambda}\right)  \theta_{j}+\bigg|  \partial_{j} \vA \left(\frac{x+{\bbb \theta}t}{\lambda}\right) \bigg| \right)
\left[\bua^{\dagger} \bua+\bub^{\dagger} \bub\right].
\end{aligned}
\]
We are ready to prove the following lemma
\begin{lem}
Let $\lambda >0$, ${\bbb\theta} =(\theta_1,\cdots,\theta_n)\in \R^n$ such that $\theta_i=-(1+b)<-1$, and
\[
\vA (x)= \sum_{j=1}^n \vA _j (x), \mbox{ and } \vA _j(x)=\phi (x_j)
\]
for $\phi(s)=\tanh(s)$ and  $\phi'(s)=\sech^2(s).$ Then, the following inequality holds
\begin{equation}\label{eq:dtI_bound}
\begin{aligned}
\dt \I 
\lesssim_{b,\lambda}&~{} 
-  \sum_{j=1}^n  \int \left|  \phi'\left(\frac{x_j+ \theta_j t}{\lambda}\right)  \right|
\left[\bua^{\dagger} \bua+\bub^{\dagger} \bub\right].
\end{aligned}
\end{equation}
\end{lem}
\begin{proof}
Firstly, we notice that
\[
 \sum_{j=1}^n \partial_j \vA  = \sum_{j=1}^n \partial_j \left( \sum_{k=1}^n \phi(x_k) \right)= \sum_j  \phi'(x_j).
\]
and $\phi'>0$. 
Then, from \eqref{eq:dtI} we get
\[
\begin{aligned}
\dt \I 
\lesssim&~{} 
   \sum_{j=1}^n  \frac{\theta_j+1}{\lambda} \int \left|
   \phi'\left(\frac{x_j+ \theta_j t}{\lambda}\right) \right|
\left[\bua^{\dagger} \bua+\bub^{\dagger} \bub\right].
\end{aligned}
\]
Recalling that $\theta_j=-(b+1)$, for $b>0$, the proof concludes.
\end{proof}

\subsection{End of proof of Theorem \ref{thm:outside}}

 Now, we conclude the proof of Theorem \ref{thm:outside}.  First, we prove decay in the right-hand side region
on the coordiante $x_j$.

	Now, for $x=(x_1,\cdots,x_n)\in \R^n$, we choose 
	\begin{equation}\label{va}
	\begin{aligned}
	&\vA _j(x)= \vA _0 (x_j),\quad 
	\vA _0(s)=\frac{1}{2}\left(1+\tanh(s)\right), \\
	&  \theta_j=-(1+b_j),\quad  \tilde{\theta}_j=-(1+\tilde{b}_j), ~{}j\in\{1,\dots, n\}.
	\end{aligned}
	\end{equation}
	with $b_j>0$ and $\tilde{b}_j=b_j/2$ and $\bbb{\theta}=(\theta_1,\cdots, \theta_n)$, $\tilde{\bbb{\theta}}=(\tilde{\theta}_1,\dots,\tilde{\theta}_n)$. 
	
	Let $\mathcal{J}^j_{t_0}(t)$ be the modified energy functional, defined for $t\in[2,t_0]$ and  $j\in \{1,\dots, n\}$,
		\begin{align}\label{43}
			\mathcal{J}^j_{t_0}(t):= \dfrac{1}{2} \int \vA _j \left(\dfrac{x+\bbb{\theta}t_0-\bbb{\tilde{\theta}}(t_0-t)}{\lambda} \right)\left[ \bua^{\dagger} \bua+\bub^{\dagger} \bub\right] .
		\end{align}
		Note that $\theta_j<\tilde{\theta}_j<0$. From Lemma \ref{prop:dtI} and proceeding exactly as in \eqref{eq:dtI_bound}, we have
		\begin{align*}
			\frac{d }{dt}\mathcal{J}^j_{t_0}(t) \lesssim_{b,\lambda} - \int  \mbox{sech}^2\left(\dfrac{x_j+\theta_j t_0-\tilde{\theta}_j(t_0-t)}{\lambda}\right)\left[ \bua \bua+\bub\bub\right] \leq 0.
		\end{align*}
		This means that the new functional $\mathcal{J}^j_{t_0}(t)$ is decreasing in $[2,t_0]$. Therefore, we have
		\begin{align*}
			\int_{2}^{t_0} \dfrac{d}{dt} \mathcal{J}^j_{t_0}(t) dt= \mathcal{J}^j_{t_0}(t_0)-\mathcal{J}^j_{t_0}(2)\leq 0\implies  \ \mathcal{J}^j_{t_0}(t_0)\leq\mathcal{J}^j_{t_0}(2).
		\end{align*}
		On the other hand, since $\lim_{s\to -\infty} \varphi_0 (s)=0,$ we have
		\begin{align*}
			\limsup_{t\to\infty}  \int \varphi_0 \left(\dfrac{x_j-\beta t-\gamma}{\lambda}\right)\left[ \bua \bua+\bub\bub\right](\delta,x) =0,
		\end{align*}
		for $\beta,\gamma,\delta>0$ fixed. Recalling \eqref{va}, this yields
		\[
		\begin{aligned}
		0 & \leq  \int \vA _j\left( \dfrac{x+\bbb{\theta}t_0-\bbb{\tilde{\theta}}(t_0-t)}{\lambda} \right)\left[ \bua \bua+\bub\bub \right](t_0,x) \\
		& \qquad \qquad \leq  \int \vA _j\left(\dfrac{x+\bbb{\theta}t_0-\bbb{\tilde{\theta}}(t_0-t)}{\lambda}\right)\left[ \bua \bua+\bub\bub \right](2,x), 
		\end{aligned}
		\]
		which implies,
		\begin{align*}
			\limsup_{t\to\infty} \int \vA _j\left(\dfrac{x+\bbb{\theta}t_0-\bbb{\tilde{\theta}}(t_0-t)}{L}\right)\left[ \bua \bua+\bub\bub\right](t,x)dx=0.
		\end{align*}
		Given \eqref{43}, an analogous argument can be applied for the left-hand side, i.e $\ca{I}_{b}^{j}(t)$, but in this case we choose $\varphi_0(s)=\frac{1}{2}\left(1-\tanh(s)\right)$. 
		Moreover, the previous limit is valid for all $j\in \{1,\cdots, n\}$,  this concludes the proof of \eqref{decay0}.

\section{Proof of Theorem \ref{thm:massless}}\label{sec:thm1p1}

Assume $m=0$ in \eqref{eq:D_LC}. Let  $(u,v)\in C(\R,L^2\times L^2(\R;\C))$ be a global solution of \eqref{eq:D_LC}. Without loss of generality, we assume that $t\geq 10$. Let
 \begin{equation*}
  \lambda(t)=\frac{t}{\log^2 t}.
 \end{equation*}
 Note that
 \[
\frac{ \dot\lambda}{\lambda}=\frac{1}{t}\left(1-\frac{2}{\log t}\right)\sim\frac{1}{t}.
 \]
 The following result is obtained by a direct application of Lemma \ref{lem:u-v}, we shall consider the massive and massless separately.
 
 \begin{cor}\label{cor:integra}
Consider $\phi=\tanh$. Then, for $m=0$, one has
 \begin{enumerate}
 \item Integrability in time
 \[
 \int_{10}^{\infty} \frac{1}{\lambda(t)} \int_{\R} \sech^2\left(\frac{x}{\lambda(t)}\right) ( |u|^2+ |v|^2  )(t,x) dxdt \lesssim 1.
 \]
 \item Sequential decay to zero: there exists $\{t_n\}\uparrow \infty$ such that
 \[
 \lim_{n\to \infty}\int_{\R} \sech^2\left(\frac{x}{\lambda(t_n)}\right) ( |u|^2+ |v|^2  )(t_n,x)=0.
 \]
 \end{enumerate}
 \end{cor}
 
\begin{proof}
Recall \eqref{eq:dt_J} in Lemma \ref{lem:u-v} and  the fact that under the choice of weight $\phi=\tanh $ one gets
\[
\left(\frac{x}{\lambda(t)}\right)^2\phi'\left(\frac{x}{\lambda(t)}\right)\in L^{\infty}(\R),
\]
uniformly in $t\in [10,\infty)$. Then, 
 \[
 \begin{aligned}
& \left| \frac{\dot\lambda}{\lambda }\int \frac{x}{\lambda} \phi'\left(\frac{x}{\lambda}\right)( |u|^2-|v|^2)\right|
\\
&~{} \leq  \frac1t  \int \frac{|x|}{\lambda} \sech^2 \left(\frac{x}{\lambda}\right)( |u|^2+|v|^2)
\\
&~{}\leq  \frac{\log^2 t}{8t} \int \sech^2 \left(\frac{x}{\lambda}\right)( |u|^2+|v|^2)
 +\frac{2}{t\log^2 t} \int \frac{|x|^2}{\lambda^2} \sech^2\left(\frac{x}{\lambda}\right)( |u|^2+|v|^2)\\
 &~{} \leq \frac{1}{8\lambda} \int \sech^2 \left(\frac{x}{\lambda}\right)( |u|^2+ |v|^2)  +\frac{2}{t\log^2 t} \left(\sup_{s\in\mathbb R} s^2 \sech^2 s \right) \int ( |u|^2+|v|^2)
\\
 &~{}\leq  \frac{1}{8\lambda} \int \sech^2\left(\frac{x}{\lambda}\right)( |u|^2+|v|^2) + \frac{C}{t\log^2 t},
\end{aligned}
\]
for some constant $C>0$. Finally, we observe from Lemma \ref{lem:u-v} that
\[
\dt\J \leq -\frac{1}{4\lambda}  \int \phi' ( |u|^2+ |v|^2  )+ f(t),
\]
where $f(t)\in L^1([10,\infty))$. From the above identity, we obtain
\[
\int_{10}^{\infty} \frac{1}{\lambda}  \int \phi'\left( \frac{x}{\lambda}\right) ( |u|^2+ |v|^2  )dx dt <\infty,
\]
since $\frac{1}{\lambda}$ does not integrable in $[10,\infty)$, this concludes the  proof of this corollary.
\end{proof}

\subsection{End of proof Theorem \ref{thm:massless}}
From  \eqref{eq:dtI1d} on Lemma \ref{cor:dtI1d},  with, $\lambda(t)=t/\log^2 (t)$ and $\phi =\sech^4\left(\frac{\cdot}{\lambda}\right)$.  One gets 
\[
\begin{aligned}
\frac{d}{dt}\mathcal{K}(t)=& 
- \frac{\dot\lambda}{\lambda} \int \frac{x}{\lambda} \phi ' \left(\frac{x}{\lambda}\right)  [   |u|^2+ |v|^2  ]
- \frac{1}{\lambda} \int \phi '\left(\frac{x}{\lambda}\right)  [   |u|^2- |v|^2  ],
\end{aligned}
\]
which is bounded as follows:
\[
\begin{aligned}
\left|\frac{d}{dt}\mathcal{K}(t)\right| \lesssim&~{} 
\frac{1}{\lambda}  \int \sech^2 \left(\frac{x}{\lambda}\right)  [  |u|^2+ |v|^2  ].
\end{aligned}
\]
Integrating in $[t,t_k]$, we have
\[
\begin{aligned}
\left|\mathcal{K}(t)-\mathcal{K}(t_k)\right| \lesssim&~{} 
 \int_{t}^{t_k} \frac{1}{\lambda} \int \sech^2\left(\frac{x}{\lambda}\right)  [    |u|^2+ |v|^2].
\end{aligned}
\]
Sending $k$ to infinity, we have from Corollary \ref{cor:integra} that $\mathcal{K}(t_k)\to 0$ and
\[
\begin{aligned}
\left|\mathcal{K}(t) \right| \lesssim&~{} 
\int_{t}^{\infty} \frac{1}{\lambda}\int \sech^2\left(\frac{x}{\lambda}\right)  [    |u|^2+ |v|^2].
\end{aligned}
\]
The proof of \eqref{limit} concludes by sending  $t\to \infty$.

\section{Proof of Theorem \ref{thm:massless2}}\label{sec:thm1p2}

Now we consider the massive 1D case. Recall 1D Dirac for the spinor $(\psi_1,\psi_2)$ as in \eqref{eq:D_1d_new}:
\begin{equation}\label{D1}
\begin{aligned}
i\partial_t \psi_1=&~{} \partial_x \psi_{2}+m\psi_1-\Wa\\
i\partial_t \psi_2=&~{} -\partial_x \psi_{1}-m\psi_2 +\Wb.
\end{aligned}
\end{equation}
Assume $m\neq 0$ and odd data and $(\psi_1,\psi_2)$ global solution. By the hypothesis of the theorem, this parity is preserved by the flow:

\begin{lem}\label{lem:par}
Assume that $\Wa,\Wb$ are odd polynomial nonlinearities, and the initial data is odd. Then, under \eqref{Laplacian}, the assumed global solution $(\psi_1,\psi_2)$ will be odd and $i\partial_t \psi_1-\partial_x \psi_2$ and $i\partial_t \psi_2 + \partial_x \psi_1$ will be odd for all times as well.
\end{lem}

\begin{proof}
Assume \eqref{D1}. Then by a simple computation
\[
\begin{aligned}
\partial_t^2\psi_{1} = &~{} \partial_x^2\psi_{1} -m^2 \psi_1+m\Wa\\
&~{}  -\left( \partial_a W_2 \partial_x\psi_{1}+\partial_b W_2 \partial_x\bar\psi_{1}+\partial_c W_2 \partial_x\psi_{2}+\partial_d W_2 \partial_x\bar \psi_{2}  \right)\\
&~{} +  \partial_a \Wa  \partial_x\psi_{2}+ m \partial_a \Wa  \psi_{1} - \partial_a \Wa \Wa \\
&~{} - \partial_b \Wa \partial_x\bar\psi_{2} -m \partial_b \Wa \bar \psi_1 +  \partial_b \Wa \overline{\Wa}\\
&~{} -  \partial_c \Wa  \partial_x\psi_{1} - m \partial_c \Wa  \psi_{2} + \partial_c \Wa W_2 \\
&~{} + \partial_d \Wa \partial_x\bar\psi_{1} + m \partial_d \Wa \bar \psi_2 -  \partial_d \Wa \overline{W}_2. 
\end{aligned}
\]
Arranging terms, we get
\begin{equation}\label{necesario}
\begin{aligned}
\partial_t^2\psi_{1} = &~{} \partial_x^2\psi_{1}-m^2 \psi_1 +m\Wa\\
&~{} -\left( \partial_a W_2 +\partial_c \Wa \right) \partial_x\psi_{1} -\left( \partial_b W_2 - \partial_d \Wa \right) \partial_x\bar \psi_{1} \\
&~{}-\left( \partial_c W_2 -\partial_a \Wa \right) \partial_x\psi_{2} -\left( \partial_d W_2 +\partial_b \Wa \right) \partial_x\bar\psi_{2} \\
&~{} + m \partial_a \Wa  \psi_{1} - \partial_a \Wa \Wa -m \partial_b \Wa \bar \psi_1 +  \partial_b \Wa \overline{\Wa}\\
&~{}  - m \partial_c \Wa  \psi_{2} + \partial_c \Wa W_2+ m \partial_d \Wa \bar \psi_2 -  \partial_d \Wa \overline{W}_2. 
\end{aligned}
\end{equation}
By \eqref{Laplacian} (see also \eqref{Laplacian2}), one gets
\[
\begin{aligned}
\partial_t^2\psi_{1} = &~{} \partial_x^2\psi_{1}-m^2 \psi_1+m\Wa+ m \partial_a \Wa  \psi_{1} - \partial_a \Wa \Wa -m \partial_b \Wa \bar \psi_1 +  \partial_b \Wa \overline{\Wa}\\
&~{}  - m \partial_c \Wa  \psi_{2} + \partial_c \Wa W_2+ m \partial_d \Wa \bar \psi_2 -  \partial_d \Wa \overline{W}_2. 
\end{aligned}
\]
Given the odd polynomial character of the nonlinearity in this equation, we ensure that odd data is preserved through time. In the case of $\psi_2$ we get
\[
\begin{aligned}
\partial_t^2\psi_{2} = &~{} \partial_x^2\psi_{2}-m^2 \psi_2 +mW_2\\
&~{} -\left( \partial_a \Wa - \partial_c W_2 \right) \partial_x\psi_{1} -\left( \partial_b \Wa + \partial_d W_2 \right) \partial_x\bar \psi_{1} \\
&~{} -\left( \partial_c \Wa + \partial_a W_2 \right) \partial_x\psi_{2} -\left( \partial_d \Wa - \partial_b W_2 \right) \partial_x\bar\psi_{2} \\
&~{} - m \partial_a W_2  \psi_{1} + \partial_a W_2 \Wa+m\partial_b W_2\bar \psi_1 -\partial_b W_2\overline{\Wa}  \\
&~{} +m \partial_c W_2\psi_2  -\partial_c W_2 W_2 -m \partial_d W_2\bar\psi_2 +\partial_d W_2 \overline{W}_2. 
\end{aligned}
\]
By \eqref{Laplacian} again, one gets
\[
\begin{aligned}
\partial_t^2\psi_{2} = &~{} \partial_x^2\psi_{2}-m^2 \psi_2 +mW_2 - m \partial_a W_2  \psi_{1} + \partial_a W_2 \Wa+m\partial_b W_2\bar \psi_1 -\partial_b W_2\overline{\Wa}  \\
&~{} +m \partial_c W_2\psi_2  -\partial_c W_2 W_2 -m \partial_d W_2\bar\psi_2 +\partial_d W_2 \overline{W}_2. 
\end{aligned}
\]
Repeating the same argument as in the case of $\psi_1$, we conclude that $\psi_2$ and its time derivative must be odd. Finally, the fact that $i\partial_t \psi_1-\partial_x \psi_2$ and $i\partial_t \psi_2 + \partial_x \psi_1$ are odd are consequences of \eqref{D1}.
\end{proof}

\begin{rem}
It is interesting to discuss the meaning of Lemma \ref{lem:par} in the integrable case $W=|u|^2|v|^2$ \eqref{eq:D_1d}-\eqref{Tirr}. Since $\partial_a \Wa -\partial_c \Wb = \frac12(|\psi_1|^2-|\psi_2|^2)\neq 0$ in principle, the term $-\left( \partial_c \Wb -\partial_a \Wa \right) \partial_x\psi_{2}$ in \eqref{necesario} survives. This term is doubly dangerous, since first: it considers the opposite variable $\psi_2$, and second, it considers the derivative of that variable. Naturally, the huge soliton solution manifold represented by \eqref{sola} is deeply related to the existence of this additional term, and it is not present in classical NLKG models. Consequently, when asking the Cauchy-Riemann condition \eqref{Laplacian}, we are discarding the integrable and all other cases of nonlinearities where mixed spatial partial derivatives appear, which are suitable conditions for the emergence of Dirac's type solitary waves. We also recall that \eqref{Laplacian} is part of a sufficient condition, not enough to conclude (since one also needs a sort of conservation law in the problem), but it is also necessary, in the sense that the integrable case is a counterexample under the absence of  \eqref{Laplacian}.
\end{rem}
Thanks to Lemma \ref{lem:par}, we can now assume odd data for all times. Taking into consideration the decomposition $\psi_j=\phi_{j1}+i\phi_{j2}$, with $\phi_{j1}$ and $\phi_{j2}$ real-valued, and considering the definition of the Wirtinger derivatives, we have for $j=1,2$,
\begin{equation}\label{Wj1}
\partial_{\overline{\psi}_j}W=\Re \partial_{\overline{\psi}_j}W+i\Im \partial_{\overline{\psi}_j}W
:= W_{j1}+i W_{j2},
\end{equation}
and the above system reads now
\begin{equation}\label{eq:SD}
\begin{aligned}
\partial_t \phi_{11} = &~{} \partial_x\phi_{22}+m\phi_{12}-W_{12} \\
\partial_t \phi_{22} = &~{} \partial_x\phi_{11}+m\phi_{21}-W_{21} \\
\partial_t \phi_{12} = &~{}- \partial_x \phi_{21}-m\phi_{11}+W_{11} \\
\partial_t \phi_{21} = &~{}- \partial_x\phi_{12}-m\phi_{22}+W_{22}.
\end{aligned}
\end{equation}
 
To prove decay as in \eqref{limit2}, we shall introduce new virial functionals adapted to \eqref{D1}, and more specifically, to \eqref{eq:SD}.

\subsection{Virial identities I}

As usual, let $\vA $ be a smooth weight with a localized derivative. Let us consider 
\begin{equation}\label{eq:J1}
\begin{aligned}
\mathcal{J}_{1}=\int \left[ \vA  \partial_x\phi_{11} +\frac12 \vA ' \phi_{11}\right](\partial_x\phi_{22}+m\phi_{12}),
\end{aligned}
\end{equation}
and
\begin{equation}\label{eq:J2}
\begin{aligned}
\mathcal{J}_{2}=\int \left[ \vA  \partial_x\phi_{12}+\frac12 \vA ' \phi_{12}\right](\partial_x\phi_{21}+m\phi_{11}).
\end{aligned}
\end{equation}
These two functionals are adapted to \eqref{eq:SD} in the following sense: 

\begin{prop}\label{Propo6p1}
It holds that
\begin{equation}\label{J1}
\begin{aligned}
\dt\mathcal{J}_{1}
=&
-\frac12\int \left[ \vA ' (\partial_x\phi_{11})^2-\frac12 \vA ''' \phi_{11}^2\right]\\
&-\int \left[ \vA  \partial_x(W_{12})+\frac12 \vA ' W_{12}\right](\partial_x\phi_{22}+m\phi_{12})
\\
&+\int \left[ \vA  \partial_x\phi_{11}+\frac12 \vA ' \phi_{11}\right]\big(-\partial_x(W_{21}) + mW_{11} \big),
\end{aligned}
\end{equation}
and
\begin{equation}\label{J2}
\begin{aligned}
\dt\mathcal{J}_{2}=&~{}
\frac12\int \left[ \vA ' (\partial_x\phi_{12})^2-\frac12 \vA ''' \phi_{12}^2\right]
\\&+\int \left[ \vA  \partial_x (W_{11})+\frac12 \vA ' W_{11}\right](\partial_x\phi_{21}+m\phi_{11})\\
&+\int \left[ \vA  \partial_x\phi_{12}+\frac12 \vA ' \phi_{12}\right]\big(\partial_x (W_{22})-mW_{12}\big).
\end{aligned}
\end{equation}
\end{prop}

\begin{proof}
Let us focus on $\mathcal{J}_1$. The case for $\mathcal J_2$ is very similar after using \eqref{eq:J2} and the same argument as below. Taking derivative on \eqref{eq:J1} with respect of time and using \eqref{eq:SD}, we obtain
\[
\begin{aligned}
\dt \mathcal{J}_{1}
=&\int \left[ \vA  \partial_x \partial_t \phi_{11}+\frac12 \vA ' \partial_t \phi_{11}\right](\partial_x\phi_{22}+m\phi_{12})
\\
&+\int \left[ \vA  \partial_x\phi_{11}+\frac12 \vA ' \phi_{11}\right]( \partial_x \partial_t \phi_{22}+m\partial_t \phi_{12})
\\
=&\int \left[ \vA  \partial_x (\partial_x\phi_{22}+m\phi_{12}-W_{12}) +\frac12 \vA ' (\partial_x\phi_{22}+m\phi_{12}-W_{12})\right](\partial_x\phi_{22}+m\phi_{12})
\\
&+\int \left[ \vA  \partial_x\phi_{11}+\frac12 \vA ' \phi_{11}\right] \\
&~{} \qquad \big(\partial_x^2 \phi_{11}+m\partial_x\phi_{21}-\partial_x(W_{21}) -m\partial_x\phi_{21}-m^2\phi_{11}+ mW_{11} \big)
\\
=&\int \left[ \vA  (\partial_x\phi_{22}+m\phi_{12}-W_{12})_{x}+\frac12 \vA ' (\partial_x\phi_{22}+m\phi_{12}-W_{12})\right](\partial_x\phi_{22}+m\phi_{12})
\\
&+\int \left[ \vA  \partial_x\phi_{11}+\frac12 \vA ' \phi_{11}\right]\big(\partial_x^2 \phi_{11}-m^2\phi_{11}-\partial_x(W_{21}) + mW_{11} \big).
\end{aligned}
\]
Since $\int \left[ \vA  f_{x}+\frac12 \vA ' f\right]f=0 $ for any $f$ localized, and integrating by parts, we get
 \[
\begin{aligned}
\dt \mathcal{J}_{1}
=&\int \left[ \vA  \partial_x(-W_{12})+\frac12 \vA ' (-W_{12})\right](\partial_x\phi_{22}+m\phi_{12})
\\
&+\int \left[ \vA  \partial_x\phi_{11}+\frac12 \vA ' \phi_{11}\right]\big(\partial_x^2 \phi_{11}-\partial_x(W_{21}) + mW_{11} \big)\\
=&
-\frac12\int \left[ \vA ' (\partial_x\phi_{11})^2-\frac12 \vA ''' \phi_{11}^2\right]\\
&-\int \left[ \vA  \partial_x(W_{12})+\frac12 \vA ' W_{12}\right](\partial_x\phi_{22}+m\phi_{12})
\\
&+\int \left[ \vA  \partial_x\phi_{11}+\frac12 \vA ' \phi_{11}\right]\big(-\partial_x(W_{21}) + mW_{11} \big)
.
\end{aligned}
\]
This proves \eqref{J1}.

\end{proof}

Let us consider now the functionals
\begin{equation}\label{J3}
\begin{aligned}
\mathcal{J}_{3}=\int \left[ \vA  \partial_x\phi_{22} +\frac12 \vA ' \phi_{22}\right](\partial_x\phi_{11}+m\phi_{21}),
\end{aligned}
\end{equation}
and
\begin{equation}\label{J4}
\begin{aligned}
\mathcal{J}_{4}=\int \left[ \vA  \partial_x\phi_{21}+\frac12 \vA ' \phi_{21}\right](\partial_x\phi_{12}+m\phi_{22}).
\end{aligned}
\end{equation}
These two functionals are adapted to \eqref{eq:SD} in the following sense: 

\begin{prop}
It holds that
\begin{equation}\label{dJ3}
\begin{aligned}
\dt \mathcal{J}_{3} = &~ {} -\frac12\int \left[ \vA ' (\partial_x\phi_{22})^2-\frac12 \vA ''' \phi_{22}^2\right]\\
&- \int \left[ \vA  \partial_x(W_{21})+\frac12 \vA ' W_{21}\right](\partial_x\phi_{11}+m\phi_{21})
\\
&+\int \left[ \vA  \partial_x\phi_{22}+\frac12 \vA ' \phi_{22}\right]\big(-\partial_x(W_{12}) + mW_{22} \big),
\end{aligned}
\end{equation}
and
\begin{equation}\label{dJ4}
\begin{aligned}
\dt \mathcal{J}_{4} =&~{}
\frac12\int \left[ \vA ' (\partial_x\phi_{21})^2-\frac12 \vA ''' \phi_{21}^2\right]
\\&+\int \left[ \vA  \partial_x (W_{22})+\frac12 \vA ' W_{22}\right](\partial_x\phi_{12}+m\phi_{22})\\
&+\int \left[ \vA  \partial_x\phi_{21}+\frac12 \vA ' \phi_{21}\right]\big(\partial_x (W_{11})-mW_{21}\big).
\end{aligned}
\end{equation}
\end{prop}
\begin{proof}
The proof of this result is similar to that in Proposition \ref{Propo6p1}. See Appendix \ref{A} for a detailed proof.
\end{proof}

\begin{cor}\label{corobueno}
Let $\varphi$ be an odd smooth function.
One has
\[
\begin{aligned}
& \dt (\mathcal{J}_1-\mathcal{J}_{2} +\mathcal J_3 -\mathcal J_4)\\
& ~{}=  -\frac12\int  \vA '  \left[ (\partial_x\phi_{11})^2 +(\partial_x\phi_{12})^2 +(\partial_x\phi_{21})^2 +(\partial_x\phi_{22})^2 \right]
\\
&~{} + \frac14\int  \vA ''' \left[  \phi_{11}^2  + \phi_{12}^2 +\phi_{21}^2 +\phi_{22}^2  \right]
\\
& +  2\sum_{i,j=1,2}\int  \vA\partial_x\phi_{ij}  W_{ij} + \int \vA ' W_{ij}\phi_{ij}.
\end{aligned}
\]
\end{cor}

\begin{proof}
First of all, one has from \eqref{J1}, \eqref{J2}, \eqref{J3} and \eqref{J4},
\[
\begin{aligned}
& \dt (\mathcal{J}_1-\mathcal{J}_{2} +\mathcal J_3 -\mathcal J_4)\\
=&
-\frac12\int  \vA '  \left[ (\partial_x\phi_{11})^2 +(\partial_x\phi_{12})^2 +(\partial_x\phi_{21})^2 +(\partial_x\phi_{22})^2 \right]
\\
&~{}+ \frac14\int  \vA ''' \left[  \phi_{11}^2  + \phi_{12}^2 +\phi_{21}^2 +\phi_{22}^2  \right]
\\& + m A -B,
\end{aligned}
\]
where
\[
\begin{aligned}
A
=& -\int \left(\vA  \partial_x W_{12} +\frac12\vA ' W_{12} \right)\phi_{12}+\int \left(\vA  \partial_x W_{11} +\frac12\vA ' W_{11} \right)\phi_{11}\\
& - \int \left(\vA  \partial_x W_{21} +\frac12\vA ' W_{21} \right)\phi_{21} +\int \left(\vA  \partial_x W_{22} +\frac12\vA ' W_{22} \right)\phi_{22}.
\end{aligned}
\]
After integrating by parts, one gets
\[
\begin{aligned}
A
=& ~{}2\sum_{i,j=1,2}\int  \vA\partial_x\phi_{ij}  W_{ij} +\sum_{i,j=1,2} \int \vA ' W_{ij}\phi_{ij} .
\end{aligned}
\]
On the other hand,
\[
\begin{aligned}
B
=& \int \left(\vA  \partial_x W_{12} +\frac12\vA ' W_{12} \right) \partial_x \phi_{22} +\int \partial_xW_{21} \left(\vA  \partial_x\phi_{11} +\frac12\vA ' \phi_{11} \right) \\
&~{} + \int \left(\vA  \partial_x W_{11} +\frac12\vA ' W_{11} \right) \partial_x \phi_{21} +\int \partial_xW_{22} \left(\vA  \partial_x\phi_{12} +\frac12\vA ' \phi_{12} \right) \\
&~{} + \int \left(\vA  \partial_x W_{21} +\frac12\vA ' W_{21} \right) \partial_x \phi_{11} +\int \partial_xW_{11} \left(\vA  \partial_x\phi_{21} +\frac12\vA ' \phi_{21} \right) \\
&~{} + \int \left(\vA  \partial_x W_{22} +\frac12\vA ' W_{22} \right) \partial_x \phi_{12} +\int \partial_xW_{12} \left(\vA  \partial_x\phi_{22} +\frac12\vA ' \phi_{22} \right) 
\\
=&~{} 2\int  \vA  \bigg( \partial_x(W_{12})  \partial_x\phi_{22}
+\partial_x (W_{22}) \partial_x\phi_{12}
+  \partial_x (W_{21}) \partial_x\phi_{11}
+  \partial_x(W_{11})  \partial_x\phi_{21} \bigg)\\
&~{} -\frac12\int \vA '' \left( W_{12}\phi_{22} +W_{11}\phi_{21}+W_{21}\phi_{11}+W_{22}\phi_{12} \right).
\end{aligned}
\]
The proof concludes by noting  that $\varphi$, $W_{ij}$ and $\phi_{jk}$ are odd functions, which implies that $B=0$.
\end{proof}

For simplicity of notation, denote
\[
\mathcal J:= \mathcal{J}_1-\mathcal{J}_{2} +\mathcal J_3 -\mathcal J_4,
\]
\[
 |\nabla \phi|^2:= (\partial_x\phi_{11})^2 +(\partial_x\phi_{12})^2 +(\partial_x\phi_{21})^2 +(\partial_x\phi_{22})^2,
\]
and
\[
|\phi|^2:= \phi_{11}^2  + \phi_{12}^2 +\phi_{21}^2 +\phi_{22}^2.
\]

Let $L>0$ be a large constant. We choose now $\vA $ as 
$\vA  =\vA _L (x)=L\vA _0(x/L)$, with $ \vA_0(s)=\tanh s $,  so that from \eqref{corobueno} one has
\begin{equation}\label{NLT}
\begin{aligned}
 \dt \mathcal J&\leq~{} -\frac12\int \vA '  |\nabla \phi|^2 + \frac14\int  \vA ''' |\phi|^2 
 \\
 &~{}\quad   + m\sum_{i,j=1,2} \left(2 \int  \vA\partial_x\phi_{ij}  W_{ij} + \int \vA ' W_{ij}\phi_{ij} \right).
\end{aligned}
\end{equation}
 Let us consider now, for a certain odd function $\phi=\phi(x)$, the quadratic quantity 
 \[
B(\phi):= \frac12\int \vA ' (\partial_x\phi)^2 + \frac14\int  \vA ''' \phi^2. 
 \]
By defining $z:=\zeta_L \phi$, with $\zeta_L := \sqrt{\vA_L '} = \sech(x/L)$, one has that $B_0(\phi)$ satisfies \cite{KMM_Nonexistence_KG}
\[
B(\phi) = \int z_{x}^2     + \frac12 \int \left(\frac{\zeta_L''}{\zeta_L}-\frac{\zeta_L'^2}{\zeta_L^2} \right) z^2 =  \int z_{x}^2- \frac1{2L} \int \sech^2 \left( \frac{x}L \right) z^2 =: B_0(z).
\]
Now we recall a coercivity estimate, based on Lemma 2.1 in \cite{KMM_Nonexistence_KG}. 
\begin{lem}
There exists $c_0>0$ such that for all $z$ odd,
\begin{equation}\label{coer}
B_0(z)  \ge c_0 \int  \left( z_{x}^2 +\frac1{L}  \sech^4 \left( \frac{x}L\right) z^2\right).
\end{equation}
\end{lem}
Now we bound the nonlinear term $mA$ reflected in \eqref{NLT}. For this, recall that $W_{ij}$ satisfies
\eqref{Wj1}. Rewriting the first term on $mA$, we obtain
\[
\begin{aligned}
&  \sum_{i,j=1,2} 2 \int  \vA_L\partial_x\phi_{ij}  W_{ij} 
\\
&\quad = \int  \vA_L(\partial_x\phi_{11}  W_{11}+\partial_x\phi_{12}  W_{12}+\partial_x\phi_{21}  W_{21}+\partial_x\phi_{22}  W_{22} )
\\
 &\quad = \frac14 \int  \vA_L\left( \partial_x(\psi_1+\overline{\psi}_1)  \left( \partial_{\overline{\psi}_1}W+\overline{\partial_{\overline{\psi}_1}W}\right) 
-\partial_x(\psi_1-\overline{\psi}_1) \left(\partial_{\overline{\psi}_1}W-\overline{\partial_{\overline{\psi}_1}W} \right) \right)
\\
 &\qquad
+ \frac14\int  \vA_L \left( \partial_x(\psi_2+\overline{\psi}_2) \left( \partial_{\overline{\psi}_2}W+\overline{\partial_{\overline{\psi}_2}W}\right) -\partial_x(\psi_2-\overline{\psi}_2) \left(\partial_{\overline{\psi}_2}W-\overline{\partial_{\overline{\psi}_2}W}\right)  \right)
\\
 &\quad =\frac12 \int  \vA_L\left( \overline{\partial_x \psi}_1  \partial_{\overline{\psi}_1}W+\partial_x \psi_1 \overline{\partial_{\overline{\psi}_1}W} \right)
+ \frac12\int  \vA_L\left( \overline{\partial_x \psi}_2 \partial_{\overline{\psi}_2}W+\partial_x\psi_2 \overline{\partial_{\overline{\psi}_2}W}\right)  
\\
 &\quad=\Re \int  \vA_L\left( \overline{\partial_x \psi}_1  \partial_{\overline{\psi}_1}W+\overline{\partial_x \psi}_2 \partial_{\overline{\psi}_2}W\right)  .
\end{aligned}
\]

Taking in mind \eqref{W_dependence}, one has $ \overline{\partial_x \psi}_1  \partial_{\overline{\psi}_1}W+\overline{\partial_x \psi}_2 \partial_{\overline{\psi}_2}W =\partial_x W$. Then
\[
\begin{aligned}
|mA|\lesssim& \left|\Re \int \vA_L' W \right|  +\sum_{i,j=1,2} \left|\int \vA' W_{ij}\phi_{ij}  \right|
\lesssim   \int \vA_L' \left(|W|+\sum_{i,j}| W_{ij}\phi_{ij} |\right).
\end{aligned}
\]

Consequently, using the fact that $W$ is of polynomial type, and $z:= \zeta_L |\phi|,$ one has for some $p\in \mathbb N$, $p\geq 3$,

\[
\begin{aligned}
|mA|& \lesssim    \int \sech^2\left( \frac{x}{L}\right)  |\phi|^{p} \lesssim  \int_0^\infty e^{x/L} |z|^3.
\end{aligned}
\]
Now the procedure is standard. Integrating by parts,
\begin{align*}
\int_0^{+\infty} e^{\frac x{L}} |z|^3
&\leq  - L \int_0^{+\infty} e^{\frac x{L}} \partial_x(|z|^3) = - 3L \int_0^{+\infty} e^{\frac x{L}} (\partial_x z) z |z|\\
&\leq 6L \| \phi \|_{L^\infty}^{\frac 12} \int_0^{+\infty} e^{\frac x{2L}} |\partial_x z| |z|^{\frac 32}\\
& \leq C_\varepsilon  L^2 \| \phi \|_{L^\infty} \int_0^{+\infty} |\partial_x z|^2 +\varepsilon  \int_0^{\infty} e^{\frac x{L}} |z|^3.
\end{align*}
Thus,
\[
\int_0^{+\infty} e^{\frac x{L}} |z|^3 \lesssim  \|\phi \|_{L^\infty} \|\partial_x z\|_{L^2}^2.
\]

Using that $\sup_{t\ge 0}\|\phi\|_{L^\infty} \leq C \delta \ll1$, and by making $c_0>0$ smaller if necessary, we get from \eqref{coer}
\[
\begin{aligned}
& \dt \mathcal J\leq -c_0 \left(  \int \sech^2\left( \frac{x}{L}\right)   |\nabla \phi|^2 +  \int  \sech^4\left( \frac{x}{L}\right)  |\phi|^2 \right) .
\end{aligned}
\]
Notice that by hypothesis, for $j=1,2,3,4$,
\[
\left| \mathcal{J}_{j}\right| \lesssim \sup_{t\geq 0} \|\phi \|^2_{H^1} \lesssim 1.
\]
Consequently, following similar estimates as in \cite{KMM_Nonexistence_KG},
\begin{equation}\label{11un}
\begin{aligned}
\int_0^\infty \left(  \int \sech^2\left( \frac{x}{L}\right)   |\nabla \phi|^2 +  \int  \sech^4\left( \frac{x}{L}\right)  |\phi|^2 \right) \lesssim 1.
\end{aligned}
\end{equation}
Let
\[
\mathcal H:= \frac12\int \sech\left(x\right) |\phi|^2 .
\]
Then
\[
\frac{d}{dt}{\mathcal  H} = \Im \int \sech (x) (\bar\psi_1 \partial_x \psi_2-\bar\psi_2 \partial_x \psi_1) +\Im \int \sech(x) \left( W_2\bar \psi_2 -W_1 \bar\psi_1\right).
\]
Consequently,
\begin{equation}\label{Ktwo}
\left| \frac{d}{dt}{\mathcal  H} \right| 
\lesssim \int \sech\left(x\right)\left[  |\nabla \phi|^2 +  |\phi|^2 \right]  . 
\end{equation} 
From \eqref{11un} there exists a sequence $t_n\to +\infty$ such that $\mathcal H(t_n)\to 0$. Let $t\in \mathbb R$. Integrating \eqref{Ktwo} on $[t,t_n]$ and passing to the limit as $n\to +\infty$ we obtain
\[
\mathcal H(t) \lesssim \int_{t}^{+\infty} \int \sech\left(x\right)\left[  |\nabla \phi|^2 +  |\phi|^2 \right]   dt.
\]
From \eqref{11un} it follows that $\lim_{t\to +\infty} \mathcal H(t)=0.$ The same holds for $t\to -\infty$. This ends the proof of the Theorem \ref{thm:massless2} in the $L^2$ case. The $L^\infty$ case is obtained by interpolation using the $L^\infty_tH_x^1$ boundedness.

\section{3D Dirac equation}\label{Section:3D}

In this section we prove Theorem \ref{thm:3D}.  Let us consider global radial solutions to \eqref{eq:PW3}. Decompose $\phi_{j}=\Re \phi+i\Im \phi=:\phi_{j1}+i\phi_{j2}$. Rewriting the real and imaginary parts of the above system, and rearranging equations, we obtain the following complex system
or corresponding  real-valued $4\times 4$ system
\begin{equation}\label{eq:SD3}
\begin{aligned}
\partial_t \phi_{11} = &~{} \Orm\phi_{22}+m\phi_{12}-W_{12} \\
\partial_t \phi_{22} = &~{} \partial_r\phi_{11}+m\phi_{21}-W_{21} \\
\partial_t \phi_{12} = &~{}- \Orm \phi_{21}-m\phi_{11}+W_{11} \\
\partial_t \phi_{21} = &~{}- \partial_r\phi_{12}-m\phi_{22}+W_{22}.
\end{aligned}
\end{equation}
To prove decay as in \eqref{limit2}, we shall introduce new virial functionals adapted to \eqref{eq:SD3}. These functional are reminiscent of previous works on NLKG models and wave maps, see e.g. \cite{KMM_Nonexistence_KG,AleMau,MoMu}.

\subsection{Virial Identities II}

Let us consider the modified virials functionals adapted to the 3D case:
\begin{equation}\label{eq:K1}
\begin{aligned}
\mathcal{K}_{1}=\int \left[ \vA  \partial_r \phi_{11} +\frac12 \vA ' \phi_{11} \right]\left( \Orm \phi_{22} + m \phi_{12} \right),
\end{aligned}
\end{equation}
and
\begin{equation}\label{eq:tK1}
\begin{aligned}
\widetilde{\mathcal{K}_{1}}=\int \left[ \vA  \partial_r \phi_{22} +\frac12 \vA ' \phi_{22} \right]\left( \partial_r \phi_{11} + m \phi_{21} \right).
\end{aligned}
\end{equation}

First of all, we claim, as in the 1D case, the following very useful auxiliary result:

\begin{claim}\label{Cl1}
Let $\vA $ be a smooth bounded function such that $\vA (0)=0$ and let $f\in H^1$. It holds that
\[
\begin{aligned}
& \int \left[ \vA  \partial_r f +\frac12 \vA ' f \right]   \left(\partial_r+\frac{2}{r} \right)\partial_r f \\
&\quad = \int \left[ \vA  \partial_r f +\frac12 \vA ' f \right]   \left( \partial_r^2 f+\frac{2}{r} \partial_r f \right)\\
&\quad = -\frac12 \int    \bigg( 2\vA '- \frac{4\vA  }{r} \bigg)(\partial_r f)^2
-\frac12  \int \bigg(  \frac{\vA ''}{r}-\frac{\vA ' }{r^2}  -\frac{\vA ''' }{2}    \bigg) f^2,
\end{aligned}
\]
and
\[
\begin{aligned}
\int \left[ \vA  \partial_r f +\frac12 \vA ' f \right]  \partial_r \left(\partial_r+\frac{2}{r}\right)f
=& \int \left[ \vA  \partial_r f +\frac12 \vA ' f \right]   \left( \partial_r^2 f+\frac{2}{r} \partial_r f -\frac{2}{r^2}f\right)\\
=& -\frac12 \int    \bigg( 2\vA '- \frac{4\vA  }{r} \bigg)(\partial_r f)^2 \\
& -\frac12  \int \bigg(  \frac{\vA ''}{r}-\frac{\vA ' }{r^2}  -\frac{\vA ''' }{2}  -2\frac{\vA }{r^3}  \bigg) f^2.
\end{aligned}
\]
\end{claim}

We can now prove the following virial identities.
\begin{prop}\label{prop7p2}
Under the assumptions on $\vA $ it holds that
\[
\begin{aligned}
\dt \mathcal{K}_1 =&~{}
- \int    \bigg( \vA '- \frac{2\vA  }{r} \bigg) (\partial_r \phi_{11})^2
\\& -\frac12  \int \bigg(  \frac{\vA ''}{r}-\frac{\vA ' }{r^2}  -\frac{\vA ''' }{2}    \bigg) \phi_{11}^2
\\&-\int \left[ \vA   \partial_r W_{12}  +\frac12 \vA ' W_{12}  \right]\left( \Orm \phi_{22} + m \phi_{12} \right)
\\
&-\int \left[ \vA  \partial_r \phi_{11} +\frac12 \vA ' \phi_{11} \right]\left( \Orm W_{21} -mW_{11} \right),
\end{aligned}
\]
and
\[
\begin{aligned}
\dt \widetilde{ \mathcal{K}_1}  
=&~{}
-  \int    \bigg( \vA '- 2\frac{\vA  }{r} \bigg)(\partial_r \phi_{22})^2
\\& -\frac12  \int \bigg(  \frac{\vA ''}{r}-\frac{\vA ' }{r^2}  -\frac{\vA ''' }{2}  -2\frac{\vA }{r^3}   \bigg) \phi_{22}^2
\\&-\int \left[ \vA  \partial_r  W_{21} +\frac12 \vA '  W_{21} \right]\left( \partial_r \phi_{11} + m \phi_{21} \right)
\\&- \int \left[ \vA  \partial_r \phi_{22} +\frac12 \vA ' \phi_{22} \right]\left( 
\partial_r W_{12} - mW_{22}  \right).
\end{aligned}
\]
\end{prop}

Notice that integrals are taken in the interval $(0,\infty)$. Therefore, the condition $\vA (0)=0$ is important to avoid undesirable boundary terms. See \cite{MoMu} for a detailed discussion on this point.

\begin{proof}
Let us focus on $\mathcal{K}_1$. Taking derivative on \eqref{eq:K1} with respect of time and using \eqref{eq:PW3}, one obtain
\[
\begin{aligned}
\dt& \mathcal{K}_{1}\\
=&~{}\int \left[ \vA  \partial_t \partial_r \phi_{11} +\frac12 \vA ' \partial_t \phi_{11} \right]\left( \Orm \phi_{22} + m \phi_{12} \right)
\\
&+\int \left[ \vA  \partial_r \phi_{11} +\frac12 \vA ' \phi_{11} \right]\left( \Orm \partial_t \phi_{22} + m\partial_t \phi_{12} \right)
\\
=&~{}\int \left[ \vA   \partial_r \left(  \left(  \partial_r   +\frac{2}{r} \right) \phi_{22}+m\phi_{12}-W_{12} \right) 
+\frac12 \vA ' \left(  \left(  \partial_r   +\frac{2}{r} \right) \phi_{22}+m\phi_{12}-W_{12} \right) \right]\\
& \qquad \left( \Orm \phi_{22} + m \phi_{12} \right)
\\
&+\int \left[ \vA  \partial_r \phi_{11} +\frac12 \vA ' \phi_{11} \right]\\
&\qquad \left( \Orm (\partial_r \phi_{11} +m\phi_{21}  -W_{21}) - m \left( \left(  \partial_r   +\frac{2}{r} \right) \phi_{21}+m\phi_{11}-W_{11}\right)\right).
\end{aligned}
\]
Using that $\int \left[ \vA  f_{x}+\frac12 \vA ' f\right]f=0 $, we get
\[
\begin{aligned}
\dt \mathcal{K}_{1}
=&~{}-\int \left[ \vA   \partial_r W_{12} 
+\frac12 \vA ' W_{12}  \right]\left( \Orm \phi_{22} + m \phi_{12} \right)
\\
&+\int \left[ \vA  \partial_r \phi_{11} +\frac12 \vA ' \phi_{11} \right]\left( \Orm \partial_r \phi_{11}  -\Orm W_{21} +mW_{11} \right)
.
\end{aligned}
\]
Applying Claim \ref{Cl1}, we conclude
\[
\begin{aligned}
\dt \mathcal{K}_{1}
=&~{}
- \int    \bigg( \vA '- \frac{2\vA  }{r} \bigg) (\partial_r \phi_{11})^2
-\frac12  \int \bigg(  \frac{\vA ''}{r}-\frac{\vA ' }{r^2}  -\frac{\vA ''' }{2}    \bigg) \phi_{11}^2
\\&-\int \left[ \vA   \partial_r W_{12}  +\frac12 \vA ' W_{12}  \right]\left( \Orm \phi_{22} + m \phi_{12} \right)
\\
&-\int \left[ \vA  \partial_r \phi_{11} +\frac12 \vA ' \phi_{11} \right]\left( \Orm W_{21} -mW_{11} \right)
.
\end{aligned}
\]
Now, let us focus on $\widetilde{\mathcal{K}_1}$. Similar as before, taking derivative in \eqref{eq:tK1} with respect of time and using \eqref{eq:PW3}, we get
\[
\begin{aligned}
\dt \widetilde{\mathcal{K}_{1}}
=&\int \left[ \vA  \partial_r \partial_t \phi_{22} +\frac12 \vA ' \partial_t \phi_{22} \right]\left( \partial_r \phi_{11} + m \phi_{21} \right)
\\ &+ \int \left[ \vA  \partial_r \phi_{22} +\frac12 \vA ' \phi_{22} \right]\left( \partial_r \partial_t \phi_{11} + m \partial_t \phi_{21} \right)
\\
=&\int \left[ \vA  \partial_r (\partial_r \phi_{11} +m\phi_{21}  -W_{21}) +\frac12 \vA ' (\partial_r \phi_{11} +m\phi_{21}  -W_{21})\right]\left( \partial_r \phi_{11} + m \phi_{21} \right)
\\& + \int \left[ \vA  \partial_r \phi_{22} +\frac12 \vA ' \phi_{22} \right] \\ 
&~{} \quad \left( \partial_r \left(  \left(  \partial_r   +\frac{2}{r} \right) \phi_{22}+m\phi_{12}-W_{12}\right) +m(-\partial_r \phi_{12}-m  \phi_{22}+W_{22} ) \right).
\end{aligned}
\]
Using that $\int \left[ \vA  f_{x}+\frac12 \vA ' f\right]f=0 $ and Claim \ref{Cl1}, we obtain
\[
\begin{aligned}
\dt \widetilde{\mathcal{K}_{1}}
=&-\int \left[ \vA  \partial_r W_{21} +\frac12 \vA '  W_{21} \right]\left( \partial_r \phi_{11} + m \phi_{21} \right)
\\+& \int \left[ \vA  \partial_r \phi_{22} +\frac12 \vA ' \phi_{22} \right]\left( \partial_r \left(   \partial_r   +\frac{2}{r} \right) \phi_{22}
-\partial_r W_{12} + mW_{22}  \right)
\\
=&~{}
-  \int    \bigg( \vA '- 2\frac{\vA  }{r} \bigg)(\partial_r \phi_{22})^2
\\& 
-\frac12  \int \bigg(  \frac{\vA ''}{r}-\frac{\vA ' }{r^2}  -\frac{\vA ''' }{2}  -2\frac{\vA }{r^3}   \bigg) \phi_{22}^2
\\&-\int \left[ \vA  \partial_r  W_{21} +\frac12 \vA '  W_{21} \right]\left( \partial_r \phi_{11} + m \phi_{21} \right)
\\&- \int \left[ \vA  \partial_r \phi_{22} +\frac12 \vA ' \phi_{22} \right]\left( 
\partial_r W_{12} - mW_{22}  \right).
\end{aligned}
\]
This concludes the proof of this proposition.
\end{proof}
As in the 1D case, we shall need further virial functionals to reconstruct the full local norm of the solution. In the 3D case, we will need two additional functionals.

Define
\begin{equation}\label{eq:def_K2}
\begin{aligned}
\mathcal{K}_{2}=\int \left[ \vA  \partial_r \phi_{12}+\frac12 \vA '  \phi_{12}\right] \left(  \Orm \phi_{21}+m \phi_{11}\right),
\end{aligned}
\end{equation}
and
\begin{equation}\label{eq:def_tK2}
\begin{aligned}
\widetilde{\mathcal{K}}_{2}=\int \left[ \vA  \partial_r \phi_{21}+\frac12 \vA '  \phi_{21}\right]( \partial_r \phi_{12}+m \phi_{22}).
\end{aligned}
\end{equation}

\begin{prop}
One has
\begin{equation}\label{eq:K2}
\begin{aligned}
\dt \mathcal{K}_{2}
=&~{}
  \int    \bigg( \vA '- 2\frac{\vA  }{r} \bigg) (\partial_r \phi_{12})^2
\\
& +\frac12  \int \bigg(  \frac{\vA ''}{r}-\frac{\vA ' }{r^2}  -\frac{\vA ''' }{2}    \bigg) \phi_{12}^2
\\&+\int \left[ \vA  \partial_r  W_{11} 
+\frac12 \vA '    W_{11} \right] \left(  \Orm \phi_{21}+m \phi_{11}\right)
\\&
+\int \left[ \vA  \partial_r \phi_{12}+\frac12 \vA '  \phi_{12}\right] \left( \Orm W_{22} -mW_{12}\right)
,
\end{aligned}
\end{equation}
and
\begin{equation}\label{eq:tK2}
\begin{aligned}
\dt\widetilde{\mathcal{K}_{2}}
=&~{}
  \int    \bigg( \vA '- 2\frac{\vA  }{r} \bigg)(\partial_r \phi_{21})^2
\\
& +\frac12  \int \bigg(  \frac{\vA ''}{r}-\frac{\vA ' }{r^2}  -\frac{\vA ''' }{2}  -2\frac{\vA }{r^3}  \bigg) \phi_{21}^2
\\&
+ \int \left[ \vA  \partial_r W_{22} +\frac12 \vA '  W_{22} \right]( \partial_r \phi_{12}+m \phi_{22})
\\&
+\int \left[ \vA  \partial_r \phi_{21}+\frac12 \vA '  \phi_{21}\right]\left( \partial_r W_{11} -mW_{21} \right).
\end{aligned}
\end{equation}
\end{prop}
\begin{proof}
The proof of this result is similar to that in Proposition \ref{prop7p2}. See Appendix \ref{A} for a detailed proof.
\end{proof}

For simplicity of notation, denote
\[
\mathcal K:= \mathcal{K}_1+ \widetilde{\mathcal{K}_{1}} - \mathcal{K}_2- \widetilde{\mathcal{K}_{2}},
\]
\[
 |\nabla \phi|^2:= (\partial_r \phi_{11})^2 +(\partial_r \phi_{12})^2 +(\partial_r \phi_{21})^2 +(\partial_r \phi_{22})^2.
\]
and
\[
|\phi|^2:= \phi_{11}^2 +\phi_{12}^2 +\phi_{21}^2 +\phi_{22}^2. 
\]
\begin{cor}\label{corobueno2}
One has
\begin{equation}\label{derK}
\begin{aligned}
 \dt \mathcal K =&~{}
  \int    \bigg( \frac{2\vA  }{r} -\vA '\bigg)| \nabla\phi |^2 
+\frac12  \int \bigg(  \frac{\vA ' }{r^2}  + \frac{\vA ''' }{2}  -\frac{\vA ''}{r}  \bigg) (\phi_{11}^2+\phi_{12}^2)
\\
&~{} +\frac12  \int \bigg( 2\frac{\vA }{r^3}  +\frac{\vA ' }{r^2}  + \frac{\vA ''' }{2}  - \frac{\vA ''}{r} \bigg) (\phi_{21}^2+\phi_{22}^2)
+m A- B,
\end{aligned}
\end{equation}
with $A$ and $B$ terms defined in \eqref{def_A_3D}-\eqref{def_B_3D}.
\end{cor}

\begin{proof}
We have from Proposition \ref{prop7p2}, \eqref{eq:K2} and \eqref{eq:tK2}:
\begin{equation}\label{todos juntos}
\begin{aligned}
 \dt \mathcal K 
=&~{}
\int    \bigg(  \frac{2\vA  }{r}-\vA ' \bigg) ((\partial_r \phi_{11})^2+(\partial_r\phi_{22})^2+(\partial_r\phi_{12})^2+(\partial_r \phi_{21})^2)
\\
& +\frac12  \int \bigg(  \frac{\vA ' }{r^2}  +\frac{\vA ''' }{2}  -\frac{\vA ''}{r}  \bigg)( \phi_{11}^2+\phi_{22}^2)
\\
&
+\frac12  \int \bigg(  2\frac{\vA }{r^3}+\frac{\vA ' }{r^2} +\frac{\vA ''' }{2}   -\frac{\vA ''}{r}  \bigg)( \phi_{21}^2+\phi_{22}^2)
\\
&-\int \left[ \vA   \partial_r W_{12}  +\frac12 \vA ' W_{12}  \right]\left( \Orm \phi_{22} + m \phi_{12} \right)
\\
&-\int \left[ \vA  \partial_r \phi_{11} +\frac12 \vA ' \phi_{11} \right]\left( \Orm W_{21} -mW_{11} \right)
\\
&-\int \left[ \vA  \partial_r  W_{21} +\frac12 \vA '  W_{21} \right]\left( \partial_r \phi_{11} + m \phi_{21} \right)
\\
&- \int \left[ \vA  \partial_r \phi_{22} +\frac12 \vA ' \phi_{22} \right]\left( 
\partial_r W_{12} - mW_{22}  \right)
\\
&-\int \left[ \vA  \partial_r  W_{11} +\frac12 \vA '    W_{11} \right] \left(  \Orm \phi_{21}+m \phi_{11}\right)
\\
& -\int \left[ \vA  \partial_r \phi_{12}+\frac12 \vA '  \phi_{12}\right] \left( \Orm W_{22} -mW_{12}\right)
\\
& - \int \left[ \vA  \partial_r W_{22} +\frac12 \vA '  W_{22} \right]( \partial_r \phi_{12}+m \phi_{22})
\\
& -\int \left[ \vA  \partial_r \phi_{21}+\frac12 \vA '  \phi_{21}\right]\left( \partial_r W_{11} -mW_{21} \right).
\end{aligned}
\end{equation}
Therefore, we arrive to
\[
\begin{aligned}
 \dt \mathcal K 
=&
 \int    \bigg( \frac{2\vA  }{r} -\vA '\bigg)| \nabla\phi |^2 
+\frac12  \int \bigg(  \frac{\vA ' }{r^2}  + \frac{\vA ''' }{2}  -\frac{\vA ''}{r}  \bigg) (\phi_{11}^2+\phi_{12}^2)
\\
&~{} +\frac12  \int \bigg( 2\frac{\vA }{r^3}  +\frac{\vA ' }{r^2}  + \frac{\vA ''' }{2}  - \frac{\vA ''}{r} \bigg) (\phi_{21}^2+\phi_{22}^2)
+m A- B,
\end{aligned}
\]
where
\[
\begin{aligned}
A=&~{}
-\int \left[ \vA  \partial_r  W_{11} +\frac12 \vA '    W_{11} \right]  \phi_{11}
+\int \left[ \vA  \partial_r \phi_{11} +\frac12 \vA ' \phi_{11} \right] W_{11} 
\\&-\int \left[ \vA   \partial_r W_{12}  +\frac12 \vA ' W_{12}  \right] \phi_{12} 
+\int \left[ \vA  \partial_r \phi_{12}+\frac12 \vA '  \phi_{12}\right] W_{12}
\\&-\int \left[ \vA  \partial_r  W_{21} +\frac12 \vA '  W_{21} \right] \phi_{21}
+\int \left[ \vA  \partial_r \phi_{21}+\frac12 \vA '  \phi_{21}\right] W_{21} 
\\&- \int \left[ \vA  \partial_r W_{22} +\frac12 \vA '  W_{22} \right] \phi_{22}
+ \int \left[ \vA  \partial_r \phi_{22} +\frac12 \vA ' \phi_{22} \right] W_{22}. 
\end{aligned}
\]
Let us simplify $A$. We have
\[
\begin{aligned}
A=&~{}
-\int  \vA  \big( \partial_r  W_{11}  \phi_{11}- \partial_r \phi_{11}  W_{11} \big)
-\int  \vA  \big(  \partial_r W_{12} \phi_{12} -\partial_r \phi_{12} W_{12} \big)
\\&-\int  \vA  \big( \partial_r  W_{21}  \phi_{21} -  \partial_r \phi_{21} W_{21} \big)
- \int  \vA  \big( \partial_r W_{22} \phi_{22} -  \partial_r \phi_{22} W_{22}   \big)
\\
=&~{}
2\int  \vA   W_{11}\partial_r \phi_{11}  +\int  \vA '  W_{11}  \phi_{11}
+2\int  \vA  \partial_r \phi_{12} W_{12} +\int  \vA '  W_{12} \phi_{12}
\\&+2\int  \vA   W_{21} \partial_r \phi_{21}+\int  \vA '  W_{21} \phi_{21}
+2 \int  \vA  \partial_r \phi_{22} W_{22}   + \int  \vA 'W_{22} \phi_{22} .
\end{aligned}
\]
Therefore, we finally arrive to the simplified expression
\begin{equation}\label{def_A_3D}
\begin{aligned}
A=&~{}
2\int  \vA   \big( W_{11}\partial_r \phi_{11}  + W_{12}\partial_r \phi_{12}  +  W_{21} \partial_r \phi_{21} + W_{22} \partial_r \phi_{22} \big)
\\&
+\int  \vA '  \big( W_{11}  \phi_{11} +W_{12} \phi_{12} +  W_{21} \phi_{21} + W_{22} \phi_{22} \big).
\end{aligned}
\end{equation}
 Now we consider the term $B$. This is defined as the nonlinear terms not containing the mass term $m$ in \eqref{todos juntos}:
\[
\begin{aligned}
B=&~{}
\int \left[ \vA  \partial_r  W_{11} +\frac12 \vA '    W_{11} \right] \Orm \phi_{21}
+\int \left[ \vA  \partial_r \phi_{21}+\frac12 \vA '  \phi_{21}\right] \partial_r W_{11} 
\\&+\int \left[ \vA   \partial_r W_{12}  +\frac12 \vA ' W_{12}  \right] \Orm \phi_{22} 
+ \int \left[ \vA  \partial_r \phi_{22} +\frac12 \vA ' \phi_{22} \right] \partial_r W_{12} 
\\&+\int \left[ \vA  \partial_r  W_{21} +\frac12 \vA '  W_{21} \right] \partial_r \phi_{11}  
+\int \left[ \vA  \partial_r \phi_{11} +\frac12 \vA ' \phi_{11} \right] \Orm W_{21} 
\\&
+ \int \left[ \vA  \partial_r W_{22} +\frac12 \vA '  W_{22} \right] \partial_r \phi_{12}
+ \int \left[ \vA  \partial_r \phi_{12}+\frac12 \vA '  \phi_{12}\right]  \Orm W_{22}.
\end{aligned}
\]
First of all,
\[
\begin{aligned}
B=&~{}
2 \int \vA  \partial_r  W_{11} \partial_r \phi_{21}
+2\int  \frac{\vA }{r} \partial_r  W_{11}  \phi_{21}
\\
&~{} + \frac12 \int\vA '    (W_{11} \partial_r \phi_{21}+ \partial_r W_{11}  \phi_{21})
+  \int \frac{\vA '}{r}    W_{11} \phi_{21}
\\&
+2 \int \vA  \partial_r  W_{12} \partial_r \phi_{22}
+2\int  \frac{\vA }{r} \partial_r  W_{12}  \phi_{22}
\\
&~{} 
+ \frac12 \int\vA '    (W_{12} \partial_r \phi_{22}+  \partial_r  W_{12}  \phi_{22})
+  \int \frac{\vA '}{r}    W_{12} \phi_{22}
\\&+2 \int  \vA  \partial_r  W_{21}\partial_r \phi_{11} 
 +\frac12  \int \vA '  (W_{21} \partial_r \phi_{11}  +\partial_r W_{21}  \phi_{11} )
\\
&~{} 
+2 \int  \frac{\vA }{r} W_{21}  \partial_r \phi_{11} + \int  \frac{ \vA '  }{r} W_{21}  \phi_{11}  
\\&+2 \int  \vA  \partial_r  W_{22}\partial_r \phi_{12} 
 +\frac12  \int \vA '  (W_{22} \partial_r \phi_{12}  +\partial_r W_{22}  \phi_{12} )
 \\
&~{} 
+2 \int  \frac{\vA }{r} W_{22}  \partial_r \phi_{12} + \int  \frac{ \vA '  }{r} W_{22}  \phi_{12} .
\end{aligned}
\]
Then, after integrating by parts and using that $\vA '(0)=0$,
\[
\begin{aligned}
B=&~{}
2 \int \vA  \partial_r  W_{11} \partial_r \phi_{21}
+2\int  \frac{\vA }{r} \partial_r  W_{11}  \phi_{21}
- \frac12 \int \bigg(\vA '' -2  \frac{ \vA '  }{r}\bigg)   W_{11} \phi_{21}
\\&
+2 \int \vA  \partial_r  W_{12} \partial_r \phi_{22}
+2\int  \frac{\vA }{r} \partial_r  W_{12}  \phi_{22}
- \frac12 \int\bigg(\vA '' -2  \frac{ \vA '  }{r}\bigg)   W_{12} \phi_{22}
\\&+2 \int  \vA  \partial_r  W_{21}\partial_r \phi_{11} 
+2 \int  \frac{\vA }{r} W_{21}  \partial_r \phi_{11} 
 -\frac12  \int \bigg(\vA '' -2  \frac{ \vA '  }{r}\bigg) W_{21}  \phi_{11}  
\\&+2 \int  \vA  \partial_r  W_{22}\partial_r \phi_{12} 
+2 \int  \frac{\vA }{r} W_{22}  \partial_r \phi_{12}
 -\frac12  \int \bigg(\vA '' -2  \frac{ \vA '  }{r}\bigg) W_{22}  \phi_{12} .
\end{aligned}
\]
After some final simplifications, and using that $(\vA /r) (r=0)=0$, we arrive to 
\begin{equation}\label{def_B_3D}
\begin{aligned}
B=&~{}
2 \int \vA   \big( 
				\partial_r  W_{11} \partial_r \phi_{21}
				+ \partial_r  W_{12} \partial_r \phi_{22}
				+ \partial_r  W_{21}\partial_r \phi_{11} 
				+\partial_r  W_{22}\partial_r \phi_{12} 
			\big)
\\&~{}
+2\int  \frac{\vA }{r} \big( 
					\partial_r  W_{11}  \phi_{21}
					+ \partial_r  W_{12}  \phi_{22}
					+ W_{21}  \partial_r \phi_{11} 
					+ W_{22}  \partial_r \phi_{12}
					\big)
\\&~{}
- \frac12 \int \bigg(\vA '' -2  \frac{ \vA '  }{r}\bigg)   \big( 
												W_{11} \phi_{21}
												+ W_{12} \phi_{22}
												+ W_{21}  \phi_{11}  
												+W_{22}  \phi_{12} 
											\big)\\
= &~{} 2 \int \vA   \big( 
				\partial_r  W_{11} \partial_r \phi_{21}
				+ \partial_r  W_{12} \partial_r \phi_{22}
				+ \partial_r  W_{21}\partial_r \phi_{11} 
				+\partial_r  W_{22}\partial_r \phi_{12} 
			\big)
\\
& ~{} +2\int \frac{\vA }r \left( W_{21} \partial_r \phi_{11}+ W_{22} \partial_r \phi_{12} -W_{11} \partial_r \phi_{21} -W_{12} \partial_r \phi_{22}\right)
\\
& ~{} + \int \left( 2\frac{\vA }{r^2} -\frac12\vA '' -\frac{\vA '}r \right)\left( W_{11}\phi_{21} +W_{12}\phi_{22}\right)\\
&~{}  -\frac12 \int \bigg(\vA '' -2  \frac{ \vA '  }{r}\bigg) \left( W_{21}\phi_{11} +W_{22}\phi_{12} \right).
\end{aligned}
\end{equation}
This ends the proof of Corollary \ref{corobueno2}. 
\end{proof}
Note that, unlike in the 1D case, the term $B$ is not zero, so we must take it into account.
 Then, we recall a well-known results which will help us to deal with the involved nonlinear terms

\begin{lem}[\cite{MoMu}]\label{lem3d}
Let $u\in H^1(\R^3)$ radial. Then $u(r)\in L^2(0,\infty)$ and $r u(r)\in L^p (0,\infty)$ for all $p\in [2,\infty]$. Moreover, it holds
\[
\sup_{r\geq 0} |r u(r)|\leq C\|u\|_{H^1(\R^3)}.
\]
\end{lem}

{
First, we consider the term $A$ defined in \eqref{def_A_3D}. Unlike the 1D case, the term $A$ has a more general structure. However, the assumption \eqref{eq:WW_w}, along with Lemma \ref{lem3d}, is enough to deal with it. Indeed, using the fact that the nonlinearity satisfies \eqref{eq:WW_w},  one obtains
\[
\begin{aligned}
|A|
\lesssim &~{}
 \int  \vA |\phi|^3|\nabla \phi| +\int | \vA '| |\phi|^4
 \\
 \lesssim &~{}
 \int  \vA (|\nabla \phi| ^2|\phi|+|\phi|^{5})+\int | \vA '| |\phi|^4
  \\
 \lesssim &~{}
\delta\int  \frac{\vA}{1+r} |\nabla \phi| ^2+
 + \delta^3 \int  \frac{\vA}{(1+r)^3}|\phi|^{2}+\delta^2 \int \frac{| \vA '|}{(1+r)^2} |\phi|^2
,
\end{aligned}
\]
since $ (W_1, W_2)$ is a pure power nonlinearity.
Finally,  rearranging the terms, we obtain   
\begin{equation}\label{eq:mA_bound}
|mA| 
 \lesssim 
\delta \int  \frac{\vA}{1+r} |\nabla \phi| ^2
 +\delta^2 \int \bigg( \frac{\vA}{(1+r)^3}+ \frac{| \vA '|}{(1+r)^2} \bigg)|\phi|^2.
\end{equation}

}

Now, we will focus on the terms associated with the nonlinear terms present in Corollary \ref{corobueno2}, $B$ in \eqref{def_B_3D}. Let us consider the following decomposition of $B$,
\begin{equation}\label{B}
\begin{aligned}
B=&~{}
2 \int \vA   \big( 
				\partial_r  W_{11} \partial_r \phi_{21}
				+ \partial_r  W_{12} \partial_r \phi_{22}
				+ \partial_r  W_{21}\partial_r \phi_{11} 
				+\partial_r  W_{22}\partial_r \phi_{12} 
			\big)
\\
& ~{} +2\int \frac{\vA }r \left( W_{21} \partial_r \phi_{11}+ W_{22} \partial_r \phi_{12} -W_{11} \partial_r \phi_{21} -W_{12} \partial_r \phi_{22}\right)
\\
& ~{} + \int \left( 2\frac{\vA }{r^2} -\frac12\vA '' -\frac{\vA '}r \right)\left( W_{11}\phi_{21} +W_{12}\phi_{22}\right)\\
&~{}  -\frac12 \int \bigg(\vA '' -2  \frac{ \vA '  }{r}\bigg) \left( W_{21}\phi_{11} +W_{22}\phi_{12} \right)\\
		 =: &~{} B_0+B_1+B_2 +B_3.
\end{aligned}
\end{equation}
Now we choose the particular weight function. Let us consider 
\begin{equation}\label{eq:3d_va}
\vA  =\frac{r^{3/2}}{1+r}.
\end{equation}
This choice is motivated by the strength of the quadratic terms, and the numerous nonlinear terms, not appearing in classical nonlinear Klein-Gordon estimates. Then
\begin{equation}\label{eq:3d_va'}
\begin{aligned}
& \vA '= \frac{\sqrt{r}(r+3)}{2(1+r)^2},\\
&  \frac{2\vA  }{r} -\vA ' = \frac{\sqrt{r} (1+3r)}{2(1+r)^2}, \quad   \frac{\vA  }{r}  = \frac{\sqrt{r}}{1+r},
 \end{aligned}
 \end{equation}
 and
 \[
\begin{aligned}
&\frac12 \bigg(  \frac{\vA ' }{r^2}  + \frac{\vA ''' }{2}   -\frac{\vA ''}{r}  \bigg) = \frac{15r^3 +95r^2+41r+9}{32r^{3/2}(1+r)^4};\\
& \frac12 \bigg(   2\frac{\vA }{r^3} + \frac{\vA ' }{r^2}  + \frac{\vA ''' }{2}  -\frac{\vA ''}{r}  \bigg) =\frac{47 r^3 + 191 r^2+137r+41}{32r^{3/2}(1+r)^4}.
\end{aligned}
\]
First we consider the large quantity $B_0$. Since $\sup_{t\in \mathbb R} \| \phi (t)\|_{L^\infty}\lesssim \delta$, using the estimate $|\phi| \lesssim \frac{\delta}{(1+r)}$  \cite{MoMu}, and that $|\partial_x W_{ij}|\leq 
 C |\phi|^{p-1}|\nabla\phi|$, with $p\in \bb{N}$ bigger or equal than 3, we get
\[
\begin{aligned}
|B_0|      
  \lesssim ~{}\int \vA |\phi|^{p-1} |\nabla \phi|^2
  \lesssim~{} \delta^{2}\int \frac{r^{3/2}}{(1+r)^{3}}|\nabla \phi|^2
    \lesssim~{} \delta^{2}\int \frac{\sqrt{r}}{(1+r)}|\nabla \phi|^2.
\end{aligned}
\]

{
Similarly, recalling \eqref{eq:mA_bound} along with \eqref{eq:3d_va} and \eqref{eq:3d_va'}, one gets
\[
\begin{aligned}
|mA| 
	 \lesssim &~{}
\delta \int  \frac{\vA}{1+r} |\nabla \phi| ^2
 + \delta^2 \int \bigg( \frac{\vA}{(1+r)^3}+ \frac{| \vA '|}{(1+r)^2} \bigg)|\phi|^2
 \\
 	 \lesssim &~{}
 \delta \int  \frac{r^{3/2}}{(1+r)^2} |\nabla \phi| ^2
 + \delta^2 \int \bigg( \frac{r^{3/2}}{(1+r)^4}+ \frac{\sqrt{r}(r+3)}{2(1+r)^4} \bigg)|\phi|^2
 \\
  	 \lesssim &~{}
 \delta \int  \frac{r^{3/2}}{(1+r)^2} |\nabla \phi| ^2
 + \delta^2 \int \frac{3\sqrt{r}(r+1)}{2(1+r)^4} |\phi|^2.
\end{aligned}
\]
Notice that for $B_1$, since $\vA \geq 0$ and $W=(W_1,W_2)$ satisfies \eqref{eq:WW_w}, we have
\[\begin{aligned}
 |B_1|
\leq &~{} 
2		 \int  \frac{\vA }{r} \big( 
					 | W_{11}   \partial_r \phi_{21} |
					+  | W_{12}  \partial_r  \phi_{22}|
					+ | W_{21} \partial_r  \phi_{11} |
					+  |W_{22}  \partial_r  \phi_{12}|
					\big)
\\
\lesssim&~{} 
		 \int  \frac{\vA }{r} |W||\nabla \phi|
\\
\lesssim&~{} 
		 \int  \frac{\vA }{r}( |\nabla \phi|^2|\phi|+ |\phi|^{2p-1})
\\
\lesssim&~{} 
		 \delta \int  \frac{\sqrt{r} }{(1+r)^2} |\nabla \phi|^2
		+\delta^3 \int  \frac{\sqrt{r} }{(1+r)^4}  |\phi|^2.
\end{aligned}
\]
}
Consequently, \eqref{derK} becomes
 \begin{equation}\label{derK_new}
\begin{aligned}
 \dt \mathcal K \geq &~{}
 \frac34 \int \frac{\sqrt{r} (1+r)}{2(1+r)^2} | \nabla\phi |^2 
+  \int \frac{r^3+r^2+r+1}{32r^{3/2}(1+r)^4} |\phi|^2
 \\&~{}  
+B_2 + B_3.
\end{aligned}
\end{equation}

{
Finally, let us focus on $B_2$ and $B_3$ in \eqref{B}. By using the estimate $|\phi| \lesssim \frac{\delta}{(1+r)}$  \cite{MoMu} and the fact that  $\sup_{t\in \mathbb R} \| \phi (t)\|_{L^\infty}\lesssim \delta$,  we obtain
}
\[
\begin{aligned}
&2\frac{\vA }{r^2} -\frac12\vA '' -\frac{\vA '}r =\frac{13r^2+22r+1}{8\sqrt{r} (1+r)^3}, \\
 &  -\frac12 \bigg(\vA '' -2  \frac{ \vA '  }{r}\bigg) = \frac{5r^2+22r+9}{8\sqrt{r} (1+r)^3}.
\end{aligned}
\]
Similarly to $B_1$, we obtain
\[
\begin{aligned}
|B_2|+|B_3|\lesssim &~{} \int \frac{r^2+r+1}{\sqrt{r} (1+r)^3}  |W||\phi| \\
 \lesssim &~{} \int \frac{|W||\phi| }{\sqrt{r} (1+r)}  \lesssim \int \frac{r |\phi|^4}{r^{3/2} (1+r)}\lesssim \delta^2 \int \frac{r (1+r)|\phi|^2}{r^{3/2} (1+r)^4} . 
\end{aligned}
\]

Therefore,
 \begin{equation*}
\begin{aligned}
 \dt \mathcal K \geq &~{}
 \frac34 \int \frac{\sqrt{r} }{2(1+r)} | \nabla\phi |^2  +  \int \frac{r^3+r^2+r+1}{128 r^{3/2}(1+r)^4} |\phi|^2.
\end{aligned}
\end{equation*}
From this estimate and the boundedness of $\mathcal K$ for all times, we obtain the integral bound
\begin{equation}\label{dIdt bonito}
\int_0^\infty \int \left(\dfrac{ \sqrt{r} |\nabla \phi|^2 }{(1+r)} +\dfrac{|\phi|^2}{r^{3/2}(1+r)} \right)dr \, dt \lesssim 1.
\end{equation}
Indeed, it is enough to consider bounds on $|\mathcal{K}_{1}|$  in \eqref{eq:K1} and $|\mathcal{\widetilde{K}}_{1}|$  in \eqref{eq:tK1}. Using that $\vA  =\frac{r^{3/2}}{1+r}$, $\vA '= \frac{\sqrt{r}(r+3)}{2(1+r)^2}$, one has
\[
\begin{aligned}
|\mathcal{K}_{1}| \lesssim &~{} \int \left[ \vA  |\partial_r \phi_{11}| +  \vA ' |\phi_{11}| \right] |\phi_{12}|  \\
&~{} + \int \left[ \vA  |\partial_r \phi_{11}| +  \vA ' |\phi_{11}| \right] |\partial_r \phi_{22}| +\int \frac1r \left[ \vA  |\partial_r \phi_{11}| +  \vA ' |\phi_{11}| \right] |\phi_{22}| \\
& =: \mathcal{K}_{1,1} +\mathcal{K}_{1,2} + \mathcal{K}_{1,3}+\mathcal{K}_{1,4} +\mathcal{K}_{1,5}+\mathcal{K}_{1,6}.
\end{aligned}
\]
Now we consider the regions $0<r<1$ and $r\geq 1$ separately. First of all,
\[
\begin{aligned}
|\mathcal{K}_{1,1}| \lesssim &~{} \int_{r< 1}  r^{3/2}  |\partial_r \phi_{11}| |\phi_{12}|  +\int_{r \geq 1} r^{1/2}  |\partial_r \phi_{11}| |\phi_{12}| \\
 \lesssim &~{} \| \nabla\phi_{11}\|_{L^2(B(0,1))}  \| \phi_{12} \|_{L^\infty(B(0,1))}  +\int_{r \geq 1} r  |\partial_r \phi_{11}| |\phi_{12}| \\
  \lesssim &~{} \| \nabla\phi_{11}\|_{L^2(B(0,1))}  \| \phi_{12} \|_{L^\infty(B(0,1))}  + \| \nabla\phi_{11}\|_{L^2} \|\phi_{12}\|_{L^2} \lesssim 1.
\end{aligned}
\]
Second,
\[
\begin{aligned}
|\mathcal{K}_{1,2}| \lesssim &~{}  \int_{r<1}    \sqrt{r}  |\phi_{11}| |\phi_{12}| + \int_{r \geq 1}    \frac{\sqrt{r}}{(1+r)}  |\phi_{11}| |\phi_{12}|\\
 \lesssim &~{}    \| \phi_{11}\|_{L^\infty} \|\phi_{12}\|_{L^\infty} + \int_{r \geq 1}  r  |\phi_{11}| |\phi_{12}| \\
 \lesssim &~{}    \| \phi_{11}\|_{L^\infty} \|\phi_{12}\|_{L^\infty} +  \| \phi_{11}\|_{L^2} \|\phi_{12}\|_{L^2} \lesssim 1.
\end{aligned}
\]
Third,
\[
\begin{aligned}
|\mathcal{K}_{1,3}| \lesssim &~{}  \int  \vA  |\partial_r \phi_{11}|  |\partial_r \phi_{22}| 
\\
 \lesssim &~{} \int_{r<1} \frac{r^{3/2}}{1+r}  |\partial_r \phi_{11}|  |\partial_r \phi_{22}|  + \int_{r \geq 1}   \frac{r^{3/2}}{1+r}  |\partial_r \phi_{11}|  |\partial_r \phi_{22}|   \\
 \lesssim &~{}  \int_{r < 1}  r    |\partial_r \phi_{11}|  |\partial_r \phi_{22}|  + \int_{r \geq 1}  r    |\partial_r \phi_{11}|  |\partial_r \phi_{22}|  \\
 \lesssim &~{}   \| \nabla \phi_{11}\|_{L^2} \| \nabla \phi_{12}\|_{L^2} \lesssim 1.
\end{aligned}
\]
Now, using that $r<1$,
\[
\begin{aligned}
|\mathcal{K}_{1,4}| \lesssim &~{} \int    \vA ' |\phi_{11}|  |\partial_r \phi_{22}|   \\
 \lesssim  &~{} \int_{r<1}    \frac{\sqrt{r}}{(1+r)} |\phi_{11}|  |\partial_r \phi_{22}|   + \int_{r \geq 1}   \frac{\sqrt{r}}{(1+r)} |\phi_{11}|  |\partial_r \phi_{22}|  \\
 \lesssim &~{}    \int_{r<1}   \frac1{\sqrt{r}} |\phi_{11}|  r |\partial_r \phi_{22}|   + \int_{r \geq 1} r |\phi_{11}|  |\partial_r \phi_{22}|   \\
 \lesssim &~{}   \| \phi_{11}\|_{L^\infty}  \left( \int_{r<1} \frac1{\sqrt{r}} \right)^{1/2} \| \nabla \phi_{12}\|_{L^2} +  \| \phi_{11}\|_{L^2} \|\nabla \phi_{22}\|_{L^2} \lesssim 1.
\end{aligned}
\]
Now we deal with $\mathcal{K}_{1,5}$. We have
\[
\begin{aligned}
|\mathcal{K}_{1,5}| \lesssim &~{} \int  \frac{\sqrt{r}}{1+r}  |\partial_r \phi_{11}|  |\phi_{22}| \\
\lesssim &~{} \int_{r<1}   \frac1{\sqrt{r}} |\phi_{22}|  r |\partial_r \phi_{11}|   + \int_{r \geq 1} r |\phi_{22}|  |\partial_r \phi_{11}| ,
\end{aligned}
\]
and the previous step on $|\mathcal{K}_{1,4}|$ allows one to conclude. Finally,
\[
\begin{aligned}
|\mathcal{K}_{1,6}| \lesssim &~{} \int \frac{1}{ \sqrt{r} (1+r)}  |\phi_{11}| |\phi_{22}|  \\
 \lesssim &~{}  \int_{r<1}   \frac1{\sqrt{r}} |\phi_{11}|   | \phi_{22}|   + \int_{r \geq 1} r |\phi_{11}|  | \phi_{22}| \\
\lesssim &~{} \|\phi_{11}\|_{L^\infty}   \| \phi_{22}\|_{L^\infty} \left( \int_{r<1}   \frac1{\sqrt{r}} \right)  + \| \phi_{11}\|_{L^2} \|\nabla \phi_{22}\|_{L^2} \lesssim 1. 
\end{aligned}
\]
Finally, the bound on $\widetilde{\mathcal{K}_{1}}$ is simpler and similar.

\subsection{End of proof} 
Let $\phi$ is a solution to \eqref{eq:PW3} and 
\[
\mathcal{H}(t) = \int_0^\infty \vA  |\phi|^2,
\]
with now $\vA $ slightly modified:  $\vA (r) = \dfrac{r^2}{(1+r)^4}$, so that $\vA(0)=0$. Then, we can see that 
\[ \begin{aligned}
\dfrac{d}{dt}\mathcal{H}(t) 
=& ~{} 2 \Im \int \vA  \bigg(\overline{\phi_1}\bigg(\partial_r+\frac{2}{r}\bigg) \phi_2 - \overline{\phi_2} \partial_r \phi_1\bigg)
\\
&~{} +2 \Im \int \vA  (\overline{\phi_2} W_2-\overline{\phi_1} W_1)
\\
=&~{} 2 \Im \int \vA  \bigg(\overline{\phi_1} \partial_r \phi_2 - \overline{\phi_2} \partial_r \phi_1\bigg)
\\
&~{} +4 \Im \int \frac{\vA }{r} \overline{\phi_1} \phi_2
+2 \Im \int \vA  (\overline{\phi_2} W_2-\overline{\phi_1} W_1).
\end{aligned} 
\] 
Simplifying,
\[ \begin{aligned}
\dfrac{d}{dt}\mathcal{H}(t) 
=&~{} 2 \Im \int \vA  \bigg(\overline{\phi_1} \partial_r \phi_2 + \partial_r \overline{\phi_2}  \phi_1\bigg)
\\
&~{} +2 \Im \int \bigg(\vA '-2\frac{\vA }{r} \bigg)\overline{\phi_1} \psi_2
+2 \Im \int \vA  (\overline{\phi_2} W_2-\overline{\phi_1} W_1)
\\
=&~{}
2 \Im \int \bigg(\vA '-2\frac{\vA }{r} \bigg)\overline{\phi_1} \phi_2
+2 \Im \int \vA  \left( \overline{\phi_2} W_2-\overline{\phi_1} W_1 \right).
\end{aligned} 
\] 
Using the value of $\vA $, we have $\vA '-2\frac{\vA }{r} = -\frac{4r^2}{(1+r)^5}$ and 
\[ \begin{aligned}
\left|\dfrac{d}{dt}\mathcal{H}(t) \right|  
\lesssim & \int  \dfrac{r^2|\phi|^2}{(1+r)^5} + \delta^2 \int  \dfrac{r^2|\phi|^2}{(1+r)^6}  \lesssim  \mathcal{H}(t).
\end{aligned} \]
Recall that \eqref{dIdt bonito} implies that there exists a sequence $t_n \to \infty$ such that $\mathcal{H}(t_n) \to 0$. Integrating the inequality above on $[t,t_n]$ we see that 
\[ \begin{aligned}
|\mathcal{H}(t_n)-\mathcal{H}(t)| =& \left|\int_t^{t_n}\dfrac{d}{dt}\mathcal{H}(s)ds\right| \leq  \int_t^{t_n}\left|\dfrac{d}{dt}\mathcal{H}(s)ds\right| \lesssim  \int_t^{t_n}\mathcal{H}(s)ds ,
\end{aligned} \]
and passing to limit as $n \to \infty$ we have 
\[
\mathcal{H}(t) \leq \int_t^\infty \mathcal{H}(s)ds,
\]
and hence $\displaystyle\lim_{t \to \infty}\mathcal{H}(t) =0$. 
To conclude the proof is enough to note that for any $R>0$ we have  
\[ \begin{aligned}
\|\phi\|_{ L^2(B(0,R))}^2 
 \leq &~{} 4\pi(1+R)^4\int_0^R \dfrac{r^2}{(1+r)^4} |\phi|^2 dr \\
\leq & ~{}4\pi(1+R)^4\mathcal{H}(t),
\end{aligned} \]
and the result \eqref{limit3D} follows.

\appendix

\section{Computations of virial identities}\label{A}

We resume here the main computations of virial identities in the 1D and 3D cases.

\medskip

Proof of \eqref{J2}. Taking derivative on \eqref{eq:J2},  we obtain
\[
\begin{aligned}
\dt \mathcal{J}_{2}=&
\int \left[ \vA  \partial_x\phi_{12}+\frac12 \vA ' \phi_{12}\right](  \partial_t\partial_x \phi_{21}+m\partial_t \phi_{11})
\\
&~{} +\int \left[ \vA   \partial_t\partial_x \phi_{12}+\frac12 \vA ' \partial_t \phi_{12}\right](\partial_x\phi_{21}+m\phi_{11}).
\end{aligned}
\]
Using \eqref{eq:SD}, we obtain
\[
\begin{aligned}
\dt \mathcal{J}_{2}=&
\int \left[ \vA  \partial_x\phi_{12}+\frac12 \vA ' \phi_{12}\right] \\
&~{} \quad \times \big(- \partial_x^2 \phi_{12}-m\partial_x\phi_{22}+\partial_x (W_{22})+m\partial_x\phi_{22}+m^2\phi_{12}-mW_{12}\big)
\\&+\int \left[ \vA  \partial_x (-\partial_x\phi_{21}-m\phi_{11}+W_{11}) +\frac12 \vA ' (-\partial_x\phi_{21}-m\phi_{11}+W_{11})\right]\\
&~{} \quad \times (\partial_x\phi_{21}+m\phi_{11}).
\end{aligned}
\]
Using that $\int \left[ \vA  f_{x}+\frac12 \vA ' f\right]f=0 $, integrating by parts, and rearranging the terms, we obtain
\[
\begin{aligned}
\dt \mathcal{J}_{2}=&
\frac12\int \left[ \vA ' (\partial_x\phi_{12})^2-\frac12 \vA ''' \phi_{12}^2\right]
\\&+\int \left[ \vA  \partial_x\phi_{12}+\frac12 \vA ' \phi_{12}\right]\big(\partial_x (W_{22})-mW_{12}\big)
\\&+\int \left[ \vA  \partial_x (W_{11})+\frac12 \vA ' W_{11}\right](\partial_x\phi_{21}+m\phi_{11}).
\end{aligned}
\]
This ends the proof of \eqref{J2}.

\medskip

Proof of \eqref{dJ3}. From \eqref{J3} and \eqref{eq:SD} one has
\[
\begin{aligned}
\dt \mathcal{J}_{3}
=&\int \left[ \vA  \partial_x \partial_t \phi_{22}+\frac12 \vA ' \partial_t \phi_{22}\right](\partial_x\phi_{11}+m\phi_{21})
\\
&+\int \left[ \vA  \partial_x\phi_{22}+\frac12 \vA ' \phi_{22}\right]( \partial_x \partial_t \phi_{11}+m\partial_t \phi_{21})
\\
=&\int \left[ \vA  \partial_x (\partial_x\phi_{11}+m\phi_{21}-W_{21}) +\frac12 \vA ' (\partial_x\phi_{11}+m\phi_{21}-W_{21})\right]
\\
&\qquad \times(\partial_x\phi_{11}+m\phi_{21})
\\
&+\int \left[ \vA  \partial_x\phi_{22}+\frac12 \vA ' \phi_{22}\right]
\\
&\qquad \times\big(\partial_x^2 \phi_{22}+m\partial_x\phi_{12}-\partial_x(W_{12}) -m\partial_x\phi_{12}-m^2\phi_{22}+ mW_{22} \big)
\\
=&\int \left[ \vA  (\partial_x\phi_{11}+m\phi_{21}-W_{21})_{x}+\frac12 \vA ' (\partial_x\phi_{11}+m\phi_{21}-W_{21})\right]
\\
& \qquad \times (\partial_x\phi_{11}+m\phi_{21})
\\
&+\int \left[ \vA  \partial_x\phi_{22}+\frac12 \vA ' \phi_{22}\right]\big(\partial_x^2 \phi_{22}-m^2\phi_{22}-\partial_x(W_{12}) + mW_{22} \big).
\end{aligned}
\]
Since $\int \left[ \vA  f_{x}+\frac12 \vA ' f\right]f=0 $ for any $f$ localized, and integrating by parts, we get
 \[
\begin{aligned}
\dt \mathcal{J}_{3}
=&\int \left[ \vA  \partial_x(-W_{21})+\frac12 \vA ' (-W_{21})\right](\partial_x\phi_{11}+m\phi_{21})
\\
&+\int \left[ \vA  \partial_x\phi_{22}+\frac12 \vA ' \phi_{22}\right]\big(\partial_x^2 \phi_{22}-\partial_x(W_{12}) + mW_{22} \big)\\
=&
-\frac12\int \left[ \vA ' (\partial_x\phi_{22})^2-\frac12 \vA ''' \phi_{22}^2\right]\\
&- \int \left[ \vA  \partial_x(W_{21})+\frac12 \vA ' W_{21}\right](\partial_x\phi_{11}+m\phi_{21})
\\
&+\int \left[ \vA  \partial_x\phi_{22}+\frac12 \vA ' \phi_{22}\right]\big(-\partial_x(W_{12}) + mW_{22} \big)
.
\end{aligned}
\]

This ends the proof of \eqref{dJ3}.

\medskip

Proof of \eqref{dJ4}. From \eqref{J4} one has
\[
\begin{aligned}
\dt \mathcal{J}_{4}=&
\int \left[ \vA  \partial_x\phi_{21}+\frac12 \vA ' \phi_{21}\right]( \partial_t\partial_x\phi_{12}+m \partial_t \phi_{22})\\
&~{}
+\int \left[ \vA   \partial_t\partial_x\phi_{21}+\frac12 \vA '  \partial_t\phi_{21}\right](\partial_x\phi_{12}+m\phi_{22}),
\end{aligned}
\]
 and using \eqref{eq:SD}, we obtain
\[
\begin{aligned}
\dt \mathcal{J}_{4}=&
\int \left[ \vA  \partial_x\phi_{21}+\frac12 \vA ' \phi_{21}\right]\\
&~{} \quad \times \big(- \partial_x^2 \phi_{21}-m\partial_x\phi_{11}+\partial_x (W_{11})+m\partial_x\phi_{11}+m^2\phi_{21}-mW_{21}\big)
\\&+\int \left[ \vA  \partial_x(-\partial_x\phi_{12}-m\phi_{22}+W_{22}) +\frac12 \vA ' (-\partial_x\phi_{12}-m\phi_{22}+W_{22})\right]\\
&~{} \quad \times (\partial_x\phi_{12}+m\phi_{22}).
\end{aligned}
\]
Using that $\int \left[ \vA  f_{x}+\frac12 \vA ' f\right]f=0 $, integrating by parts, and rearranging the terms, we obtain
\[
\begin{aligned}
\dt \mathcal{J}_{4}=&~{}
\frac12\int \left[ \vA ' (\partial_x\phi_{21})^2-\frac12 \vA ''' \phi_{21}^2\right]
\\&+\int \left[ \vA  \partial_x\phi_{21}+\frac12 \vA ' \phi_{21}\right]\big(\partial_x (W_{11})-mW_{21}\big)
\\&+\int \left[ \vA  \partial_x (W_{22})+\frac12 \vA '  W_{22}\right](\partial_x\phi_{12}+m\phi_{22}).
\end{aligned}
\]
This ends the proof of \eqref{dJ4}.

Let us prove \eqref{eq:K2}.

Taking derivative on \eqref{eq:def_K2} with respect of time and using \eqref{eq:PW3}
\[
\begin{aligned}
\dt &\mathcal{K}_{2}
\\
=&~{}
\int \left[ \vA  \partial_r \partial_t \phi_{12}+\frac12 \vA '  \partial_t \phi_{12}\right] \left(  \Orm \phi_{21}+m \phi_{11}\right)
\\&
+\int \left[ \vA  \partial_r \phi_{12}+\frac12 \vA '  \phi_{12}\right] \left(  \Orm \partial_t \phi_{21}+m \partial_t \phi_{11}\right)
\\
=&- \int \left[ \vA  \partial_r \left( \Orm \phi_{21}+m\phi_{11}-W_{11} \right)
+\frac12 \vA '   \left( \Orm \phi_{21}+m\phi_{11}-W_{11} \right)\right] \\
& \qquad \times \left(  \Orm \phi_{21}+m \phi_{11}\right)
\\&
+\int \left[ \vA  \partial_r \phi_{12}+\frac12 \vA '  \phi_{12}\right] \\
& \qquad \times \left(  \Orm(-\partial_r \phi_{12}-m  \phi_{22}+W_{22})
+m \left( \left(  \partial_r   +\frac{2}{r} \right) \phi_{22}+m\phi_{12}-W_{12}\right)\right).
\end{aligned}
\]
Using that $\int \left[ \vA  f_{x}+\frac12 \vA ' f\right]f=0 $ and applying Claim \ref{Cl1}, we conclude, we get
\[
\begin{aligned}
\dt \mathcal{K}_{2}
=&~{}
 \int \left[ \vA  \partial_r  W_{11}  +\frac12 \vA '    W_{11} \right] \left(  \Orm \phi_{21}+m \phi_{11}\right)
\\&
+\int \left[ \vA  \partial_r \phi_{12}+\frac12 \vA '  \phi_{12}\right] \left( - \Orm\partial_r \phi_{12}
+\Orm W_{22}
-mW_{12}\right)
\\
=& ~{}
  \int    \bigg( \vA '- 2\frac{\vA  }{r} \bigg) (\partial_r\phi_{12})^2
+\frac12  \int \bigg(  \frac{\vA ''}{r}-\frac{\vA ' }{r^2}  -\frac{\vA ''' }{2}    \bigg) \phi_{12}^2
\\&+\int \left[ \vA  \partial_r  W_{11} 
+\frac12 \vA '    W_{11} \right] \left(  \Orm \phi_{21}+m \phi_{11}\right)
\\&
+\int \left[ \vA  \partial_r \phi_{12}+\frac12 \vA '  \phi_{12}\right] \left( \Orm W_{22} -mW_{12}\right).
\end{aligned}
\]

\medskip
Finally,  let us prove \eqref{eq:tK2}. Similar as before, taking derivative on \eqref{eq:def_tK2} with respect of time and using \eqref{eq:PW3}
\[
\begin{aligned}
\dt \widetilde{\mathcal{K}_{2}}
=&
\int \left[ \vA  \partial_r \partial_t \phi_{21}+\frac12 \vA '  \partial_t \phi_{21}\right]( \partial_r \phi_{12}+m \phi_{22})\\
& ~{}
+\int \left[ \vA  \partial_r \phi_{21}+\frac12 \vA '  \phi_{21}\right]( \partial_r \partial_t \phi_{12}+m \partial_t \phi_{22})
\\
=&
-\int \left[ \vA  \partial_r ( \partial_r \phi_{12}+m  \phi_{22}-W_{22} )+\frac12 \vA '  ( \partial_r \phi_{12}+m  \phi_{22}-W_{22} )\right]\\
&~{}\qquad \times( \partial_r \phi_{12}+m \phi_{22})
\\&
-\int \left[ \vA  \partial_r \phi_{21}+\frac12 \vA '  \phi_{21}\right] \\
&~{} \qquad \times \left( \partial_r \left(\Orm \phi_{21}+m\phi_{11}-W_{11} \right)-m(\partial_r \phi_{11} +m\phi_{21}  -W_{21}) \right) .
\end{aligned}
\]
Using that $\int \left[ \vA  f_{x}+\frac12 \vA ' f\right]f=0 $ and Claim \ref{Cl1}, we obtain
\[
\begin{aligned}
\dt \widetilde{\mathcal{K}_{2}}
=&
  \int    \bigg( \vA '- 2\frac{\vA  }{r} \bigg)(\partial_r \phi_{21})^2
+\frac12  \int \bigg(  \frac{\vA ''}{r}-\frac{\vA ' }{r^2}  -\frac{\vA ''' }{2}  -2\frac{\vA }{r^3}  \bigg) \phi_{21}^2
\\&
+ \int \left[ \vA  \partial_r W_{22} +\frac12 \vA '  W_{22} \right]( \partial_r \phi_{12}+m \phi_{22})
\\&
+\int \left[ \vA  \partial_r \phi_{21}+\frac12 \vA '  \phi_{21}\right]\left( \partial_r W_{11} -mW_{21} \right).
\end{aligned}
\]
This concludes the proof of these identities.

\section{On the rigorous use of the NLKG trick}\label{AppB} 

In this section we discuss the validity of the NLKG formulation in the case of Dirac models. To fix ideas, we shall only consider the 1D case. Firstly we recall the classical and well-known 1D Klein-Gordon equation, given by
\[
\partial_t^2 u-\partial_x^2u=mu+f(u),  \quad u\in \R, \quad (t,x)\in \R\times \R.
\]
 The classical approach to study the above equation is to rewrite it as a system considering the variables $u_1=u$ and $u_2=\partial_t u$, obtaining the classical 'Hamiltonian' formulation. However, considering the variables $\partial_t u=\partial_x v$, one arrives to a slightly different system:
 \begin{equation}\label{eq:KG2}
 \begin{aligned}
 \partial_t u=&~{} \partial_x v\\
 \partial_t v =&~{} \partial_x u+mu+f(u),
 \end{aligned}
 \end{equation}
 This formulation resembles a simplified version of the system studied in \eqref{eq:SD}. In fact, when the solution to \eqref{eq:SD} is sufficiently regular, we observe that under \eqref{Laplacian} the system \eqref{D1} leads to 
 \begin{equation}\label{NLKG_2}
\begin{aligned}
&  (\partial_t^2-\partial_x^2+m^2) \psi_1 -m W_1 +\partial_x (W_2) \\
& \qquad -i \left( \partial_{\psi_1} W_1(-im\psi_1+iW_1) +\partial_{\overline{\psi}_1} W_1\overline{(-im\psi_1+iW_1)} \right)\\
& \qquad -i \left(  \partial_{\psi_2} W_1(im\psi_2 -iW_2)	+\partial_{\overline{\psi}_2} W_1 \overline{(im\psi_2 -iW_2)} \right)=0,\\
&(\partial_t^2-\partial_x^2+m^2 )\psi_2-m W_2 +\partial_x (W_1 ) \\
& \qquad + i\left( \partial_{\psi_1} W_2(-im\psi_1+iW_1)+\partial_{\overline{\psi}_1} W_2\overline{(-im\psi_1+iW_1)} \right)\\
&\qquad  +i \left(  \partial_{\psi_2} W_2(im\psi_2 -iW_2) +\partial_{\overline{\psi}_2} W_2 \overline{(im\psi_2 -iW_2)} \right) =0.
 \end{aligned}
\end{equation}
 Therefore, it is suggested that there is an underlying Klein-Gordon structure, as in \eqref{eq:KG2}. This naturally justifies some of the virials used along this work.
 
\medskip

However, the equivalence between the NLKG \eqref{NLKG_2} and the Dirac \eqref{D1} formulations is not direct or rigorous in every situation. Specially the converse statement: every solution to the underlying NLKG model defines a solution in the Dirac model \eqref{D1}. Here we present a rigorous result in that direction.

\begin{lem}\label{trick}
Assume that $W=(W_1,W_2)$ satisfies the harmonic condition \eqref{Laplacian}. Let $(\psi_1,\psi_2)$ be a smooth solution to \eqref{D1}. Then \eqref{NLKG_2} is satisfied. Conversely, if $(\psi_1,\psi_2)$ solves \eqref{NLKG_2}, 
\[
(\psi_1,\psi_2)(t=0)=(\partial_t \psi_1+i \partial_x \psi_{2},\partial_t \psi_2-i \partial_x \psi_{1})(t=0)=0,
\]
and $\sup_{t\geq 0} \| W'(\psi_1,\psi_2)(t)\|_{L^\infty_x} \lesssim 1$, then \eqref{D1} is satisfied in the $L^2$ sense.
\end{lem}

\begin{rem}
Here $W'$ represents the matrix derivative of $W$. The uniform condition $\sup_{t\geq 0} \| W'(\psi_1,\psi_2)(t)\|_{L^\infty_x} \lesssim 1$ can be replaced by the simpler but more demanding one $\sup_{t\geq 0} \| (\psi_1,\psi_2)(t)\|_{L^\infty_x} \lesssim 1$, which in the 1D case is implied by the natural finite energy condition $\sup_{t\geq 0} \| W'(\psi_1,\psi_2)(t)\|_{H^1_x} \lesssim 1$. 
\end{rem}

\begin{proof}
The first statement is included in the proof of Lemma \ref{lem:par}. Let us prove the converse.
 Let us consider $(\psi_1,\psi_2)$ solution to \eqref{NLKG_2}, and set
\[
\begin{aligned}
u_0 :=&~{}\partial_t \psi_1+i \partial_x\psi_{2}+im\psi_1-iW_1,\\
v_0:=&~{}\partial_t \psi_2-i \partial_x \psi_{1}-im\psi_2 +iW_2.
\end{aligned}
\]
Then, $(u_0,v_0)$ satisfies the following system
\[
\begin{aligned}
\partial_t u_0-i\partial_x v_0-imu_0=&~{} (\partial_t^2-\partial_x^2+m^2) \psi_1
				-m W_1
				+\partial_x (W_2) 
				-i\partial_t (W_1),\\
\partial_t v_0+i\partial_x u_0+im  v_0=&~{} (\partial_t^2-\partial_x^2+m^2 )\psi_2
					-m W_2
					+\partial_x (W_1 )
					+i\partial_t(W_2).
\end{aligned}
\]
Notice that
\[
\begin{aligned}
\partial_t (W_1 )
			=&~{} \partial_{\psi_1} W_1(-i\psi_{2,x}-im\psi_1+iW_1)
				\\& +\partial_{\overline{\psi}_1} W_1\overline{(-i\psi_{2,x}-im\psi_1+iW_1)}
				\\&
				+ \partial_{\psi_2} W_1(i\psi_{1,x}+im\psi_2 -iW_2)
				\\&
				+\partial_{\overline{\psi}_2} W_1 \overline{(i\psi_{1,x}+im\psi_2 -iW_2)}
				\\
				 &+\partial_{\psi_1} W_1 u_0+\partial_{\overline{\psi}_1} W_1\overline{u}_0
				+ \partial_{\psi_2} W_1 v_0 +\partial_{\overline{\psi}_2} W_1 \overline{v}_0,
\end{aligned}
\]
and
\[
\begin{aligned}
\partial_t(W_2)=&~{} \partial_{\psi_1} W_2(-i\psi_{2,x}-im\psi_1+iW_1)
				\\&
				+\partial_{\overline{\psi}_1} W_2\overline{(-i\psi_{2,x}-im\psi_1+iW_1)}
				\\&
				+ \partial_{\psi_2} W_2(i\psi_{1,x}+im\psi_2 -iW_2)
				\\&
				+\partial_{\overline{\psi}_2} W_2 \overline{(i\psi_{1,x}+im\psi_2 -iW_2)}
				\\
				 &+\partial_{\psi_1} W_2 u_0+\partial_{\overline{\psi}_1} W_2\overline{u}_0
				+ \partial_{\psi_2} W_2 v_0+\partial_{\overline{\psi}_2} W_2 \overline{v}_0.
\end{aligned}
\]
Recall that \eqref{Laplacian} is satisfied. Therefore, for $(\psi_1,\psi_2)$ solution to \eqref{NLKG_2}, we obtain that $(u_0,v_0)$ satisfies the following system
\[
\begin{aligned}
\partial_t u_0-i\partial_x v_0-imu_0=&~{} -i (\partial_{\psi_1} W_1 u_0+\partial_{\overline{\psi_1}} W_1\overline{u_0}) -i( \partial_{\psi_2} W_1 v_0+\partial_{\overline{\psi_2}} W_1 \overline{v_0}),
\\
\partial_t v_0+i\partial_x u_0+im  v_0=&~{} i	(\partial_{\psi_1} W_2 u_0+\partial_{\overline{\psi_1}} W_2\overline{u_0})
+i(\partial_{\psi_2} W_2 v_0+\partial_{\overline{\psi_2}} W_2 \overline{v_0}).
\end{aligned}
\]
This can be considered as a linear system for $(u_0,v_0)$ with variable coefficients depending on $(\psi_1,\psi_2)$. 
Now, let us consider the $L^2$ mass. We have
\[
\begin{aligned}
& \frac12 \dt  \int \left( |u_0|^2+|v_0|^2 \right)\\
		&~{} \quad=  \Re\int \left( \partial_t u_0 \overline{u}_0+ \partial_t v_0 \overline{v}_0\right)
		\\
		&~{} \quad =
		 \Re\int \left(  \overline{u_0} ( -i (\partial_{\psi_1} W_1 u_0+\partial_{\overline{\psi_1}} W_1\overline{u}_0) -i( \partial_{\psi_2} W_1 v_0+\partial_{\overline{\psi_2}} W_1 \overline{v}_0)) \right)
		\\&
			\quad \quad +\Re\int  \left( \overline{v_0} ( i	(\partial_{\psi_1} W_2 u_0+\partial_{\overline{\psi_1}} W_2\overline{u}_0) +i(\partial_{\psi_2} W_2 v_0+\partial_{\overline{\psi_2}} W_2 \overline{v}_0)) \right)
	\\
			&~{} \quad= 
		\Im \int \left(  \overline{u}_0 (  (\partial_{\psi_1} W_1 u_0+\partial_{\overline{\psi}_1} W_1\overline{u}_0) +( \partial_{\psi_2} W_1 v_0+\partial_{\overline{\psi}_2} W_1 \overline{v}_0)) \right)
		\\ & \quad \quad 
		-\Im \int \left( \overline{v}_0 ( 	(\partial_{\psi_1} W_2 u_0+\partial_{\overline{\psi}_1} W_2\overline{u}_0) +(\partial_{\psi_2} W_2 v_0+\partial_{\overline{\psi}_2} W_2 \overline{v}_0)) \right).
\end{aligned}
\]
Using the hypothesis, we obtain
\[
\begin{aligned}
\frac12 \dt  \int \left( |u_0|^2+|v_0|^2 \right) \lesssim &~{}  \| W'(\psi_1,\psi_2)(t) \|_{L^\infty_x} \int \left( |u_0|^2+|v_0|^2 \right) \\
\lesssim &~{} \int \left( |u_0|^2+|v_0|^2 \right).
\end{aligned}
\]
Since $(u_0,v_0)(t=0)=(0,0)$ by hypothesis, we conclude the desired result.
\end{proof}

\subsection*{Data Availability} All the data obtained for this work is presented in the same manuscript.

\subsection*{Conflict of Interest} The authors declare no conflict of interest in the production and possible publication of this work.

\end{document}